\documentclass[9pt]{article}
\usepackage{latexsym,amsfonts,amssymb,amsmath,amsthm}

\usepackage{color}

\begin{document}

\newtheorem{tm}{Theorem}[section]
\newtheorem{prop}[tm]{Proposition}
\newtheorem{defin}[tm]{Definition}
\newtheorem{coro}[tm]{Corollary}

\newtheorem{lem}[tm]{Lemma}
\newtheorem{assumption}[tm]{Assumption}

\newtheorem{rk}[tm]{Remark}

\newtheorem{nota}[tm]{Notation}
\numberwithin{equation}{section}

\newcommand{\stk}[2]{\stackrel{#1}{#2}}
\newcommand{\dwn}[1]{{\scriptstyle #1}\downarrow}
\newcommand{\upa}[1]{{\scriptstyle #1}\uparrow}
\newcommand{\nea}[1]{{\scriptstyle #1}\nearrow}
\newcommand{\sea}[1]{\searrow {\scriptstyle #1}}
\newcommand{\csti}[3]{(#1+1) (#2)^{1/ (#1+1)} (#1)^{- #1
 / (#1+1)} (#3)^{ #1 / (#1 +1)}}
\newcommand{\RR}[1]{\mathbb{#1}}

\newcommand{ \bl}{\color{blue}}
\newcommand {\rd}{\color{red}}
\newcommand{ \bk}{\color{black}}
\newcommand{ \gr}{\color{OliveGreen}}
\newcommand{ \mg}{\color{RedViolet}}

\newcommand{\ep}{\varepsilon}
\newcommand{\rr}{{\mathbb R}}
\newcommand{\alert}[1]{\fbox{#1}}

\newcommand{\eqd}{\sim}
\def\R{{\mathbb R}}
\def\N{{\mathbb N}}
\def\Q{{\mathbb Q}}
\def\C{{\mathbb C}}
\def\l{{\langle}}
\def\r{\rangle}
\def\t{\tau}
\def\k{\kappa}
\def\a{\alpha}
\def\la{\lambda}
\def\De{\Delta}
\def\de{\delta}
\def\ga{\gamma}
\def\Ga{\Gamma}
\def\ep{\varepsilon}
\def\eps{\varepsilon}
\def\si{\sigma}
\def\Re {{\rm Re}\,}
\def\Im {{\rm Im}\,}
\def\E{{\mathbb E}}
\def\P{{\mathbb P}}
\def\Z{{\mathbb Z}}
\def\D{{\mathbb D}}
\newcommand{\ceil}[1]{\lceil{#1}\rceil}

\title{Spreading Speeds and Traveling waves of a parabolic-elliptic chemotaxis system with  logistic source on $\mathbb{R}^N$}

\author{
Rachidi B. Salako and Wenxian Shen  \\
Department of Mathematics and Statistics\\
Auburn University\\
Auburn University, AL 36849\\
U.S.A. }

\date{}
\maketitle

\begin{abstract}
The current paper is devoted to the study of  spreading speeds and traveling wave solutions of the following parabolic-elliptic chemotaxis system,\begin{equation*}\label{IntroEq0-2}\begin{cases}u_{t}=\Delta{u}-\chi\nabla\cdot(u\nabla{v})+u(1-u),\quad{x}\in\mathbb{R}^N\\{0}=\Delta{v}-v+u,\quad{x}\in\mathbb{R}^N,\end{cases}\end{equation*}
where $u(x,t)$\,\,{}represents the population density of a mobile species and\,\,{}$v(x,t)$\,\,{}represents the population density of a chemoattractant, and\,\,{}$\chi$\,\,{}represents the chemotaxis sensitivity. We first give a detailed study in the case\,\,{}$N=1$.\,\,{}In this case, it has been shown in an earlier work by the authors of the current paper that, when\,\,{}$0<\chi<1$,\,\,{}for every nonnegative uniformly continuous and bounded function\,\,{}$u_0(x)$,\,\,{}the system has a unique globally bounded classical solution\,\,{}$(u(x,t;u_0),v(x,t;u_0))$\,\,{}with initial condition\,\,{}$u(x,0;u_0)=u_0(x)$.\,\,{}Furthermore, it was shown that, if\,\,{}$0<\chi<\frac{1}{2}$,\,\,{}then the constant steady-state solution\,\,{}$(1,1)$\,\,{}is asymptotically stable with respect to strictly positive perturbations. In the current paper, we show that if\,\,{}$0<\chi<1$,\,\,{}then there are nonnegative constants\,\,{}$c_{-}^*(\chi)\leq c_+^*(\chi)$\,\,{}such that for every nonnegative initial function\,\,{}$u_0(\cdot)$\,\,{}with non-empty and compact support\,\,{}${\rm{supp}}(u_0)$,$$\lim_{t\to\infty}\sup_{|x|\leq{ct}}\big[|u(x,t;u_0)-1|+|v(x,t;u_0)-1|\big]=0\quad\forall\,\,{0}<c<c_{-}^*(\chi)$$and$$\lim_{t\to\infty}\sup _{|x|\geq{ct}}\big[u(x,t;u_0)+v(x,t;u_0)\big]=0\quad\forall\,\,{c}>c_{+}^*(\chi).$$We also show that if $0<\chi<\frac{1}{2}$, there is a positive constant\,\,{}$c^*(\chi)$\,\,{}such that for every\,\,{}$c\geq{c}^*(\chi)$,\,\,{}the system  has a traveling wave solution\,\,{}$(u(x,t),v(x,t))$\,\,{}with speed $c$ and connecting\,\,{}$(1,1)$\,\,{}and\,\,{}$(0,0)$,\,\,{}that is,\,\,{}$(u(x,t),v(x,t))=(U(x-ct),V(x-ct))$\,\,{}for some functions\,\,{}$U(\cdot)$\,\,{}and\,\,{}$V(\cdot)$\,\,{}satisfying\,\,{}$(U(-\infty),V(-\infty))=(1,1)$\,\,{}and\,\,{}$(U(\infty),V(\infty))=(0,0)$.\,\,{}Moreover, we show that$$\lim_{\chi\to{0}}c^*(\chi)=\lim_{\chi\to{0}}c_+^*(\chi)=\lim_{\chi\to{0}}c_-^*(\chi)=2.$$We then consider the extensions of the results in the case\,\,{}$N=1$\,\,{}to the case\,\,{}$N\geq2$.
\end{abstract}



\section{Introduction}
The origin of  chemotaxis models was introduced by Keller and Segel  (see \cite{KeSe1}, \cite{KeSe2}). The following is a general Keller-Segel model
for the time evolution of both the density $u(x,t)$ of a mobile species and the density $v(x,t)$ of a  chemoattractant,
\begin{equation}\label{IntroEq0}
\begin{cases}
u_{t}=\nabla\cdot (m(u)\nabla u- \chi(u,v)\nabla v) + f(u,v),\quad   x\in\Omega \\
\tau v_t=\Delta v + g(u,v),\quad  x\in\Omega
\end{cases}
\end{equation}
complemented with certain boundary condition on $\partial\Omega$ if $\Omega$ is bounded, where $\Omega\subset \R^N$ is an open domain, $\tau\ge 0$ is a non-negative constant linked to the speed of diffusion of the chemical, the function $\chi(u,v)$ represents  the sensitivity with respect to chemotaxis,  and the functions $f$ and $g$ model the growth of the mobile species and the chemoattractant, respectively.

Among the central problems about \eqref{IntroEq0} are
 global existence of classical/weak solutions with given initial  functions;  finite-time blow-up;  pattern formation;  existence, uniqueness, and stability of certain special solutions;  spatial spreading and front propagation dynamics when the domain is a whole space; etc.

In the present paper, we restrict ourselves to  the case that $\tau=0$, which is supposed to model the situation when the chemoattractant diffuses very quickly. System \eqref{IntroEq0} with $\tau=0$ reads as
\begin{equation}\label{IntroEq0-1}
\begin{cases}
u_{t}=\nabla\cdot (m(u)\nabla u- \chi (u,v)\nabla v) + f(u,v),\quad  x\in\Omega \\
0=\Delta v + g(u,v),\quad  x\in\Omega
\end{cases}
\end{equation}
complemented with certain boundary condition on $\partial\Omega$ if $\Omega$ is bounded.

Global existence and asymptotic behavior of solutions of \eqref{IntroEq0-1} on bounded domain $\Omega$ have been extensively studied by many authors.
The reader is referred to \cite{BBTW}, \cite{DiNa}, \cite{GaSaTe}, \cite{TeWi},  \cite{WaMuZh}, \cite{Win}, \cite{win_jde}, \cite{win_JMAA_veryweak}, \cite{win_arxiv}, \cite{win_JNLS}, \cite{YoYo}, \cite{ZhMuHuTi},  and references
 therein for the studies of \eqref{IntroEq0-1} on bounded domain with Neumann or Dirichlet  boundary conditions and with $f(u,v)$ being logistic type source function or $0$ and  $m(u)$, $\chi(u,v)$, and $g(u,v)$ being  various kinds of functions.

 There are also several studies of \eqref{IntroEq0-1} when $\Omega$ is the whole space $\R^N$ and $f(u,v)=0$ (see \cite{DiNaRa}, \cite{KKAS} \cite{NAGAI_SENBA_YOSHIDA}, \cite{SuKu}, \cite{Sug}). For example,
in the  case of   $m(u)\equiv 1$, $\chi(u,v)=\chi u$,  $f(u,v)=0$, and $g(u,v)=u-v$, where  $\chi$ is a positive constant, it is known that blow-up occurs if
either N=2 and the total initial population mass is large enough, or $N\ge 3$ (see  \cite{BBTW}, \cite{DiNaRa},  \cite{NAGAI_SENBA_YOSHIDA} and  references therein).  However, there is not much study of \eqref{IntroEq0-1} when $\Omega=\R^N$
and $f(u,v)\not =0$.

In the current paper, we will study  spatial spreading and  front propagation dynamics of
 \eqref{IntroEq0-1}  with the following choices, $\Omega=\R^N$,   $m(u)=1$, $\chi(u,v)=\chi u$ with $\chi$ being a nonnegative constant,
 $f(u,v)=u(1-u)$, and $g(u,v)=u-v$. With such choices, \eqref{IntroEq0-1} becomes
 \begin{equation}\label{IntroEq0-2-new}
\begin{cases}
u_{t}=\Delta u- \chi \nabla \cdot (u \nabla v) + u(1-u),\quad x\in\R^N\cr
0=\Delta u-v+u, \quad x\in\R^N.
\end{cases}
\end{equation}
We will provide a detailed study on the spatial spreading and  front propagation dynamics of \eqref{IntroEq0-2-new}
in the case $N=1$ and then discuss the extensions of the obtained results for the case $N=1$ to $N\ge 2$.
Here are three main reasons for doing that. First, the study of traveling wave solutions on $\R^N$ reduces to the study of traveling wave solutions on $\R$. Second,  we can get some nicer results in the case $N=1$ (compare Theorem B(i) and Theorem D(i)). Third, it is for the simplicity in notations.

Consider \eqref{IntroEq0-2-new} with $N=1$, that is,
\begin{equation}\label{IntroEq0-2}
\begin{cases}
u_{t}=u_{xx}- \chi (u v_x)_x + u(1-u),\quad x\in\R\cr
0=v_{xx}-v+u, \quad x\in\R.
\end{cases}
\end{equation}
In the very recent work \cite{SaSh}, the authors of the current paper studied  the global existence of classical solutions with various given initial functions and the asymptotic behavior of global positive solutions of \eqref{IntroEq0-2} (actually, \cite{SaSh} considered a little more general system, namely, \eqref{IntroEq0-2-new} with $u(1-u)$ being replaced by $u(a-bu)$).
Let
\begin{equation}
\label{unif-cont-space}
C_{\rm unif}^b(\R)=\{u\in C(\R)\,|\, u(x)\,\,\text{is uniformly continuous in}\,\, x\in\R\,\, {\rm and}\,\, \sup_{x\in\R}|u(x)|<\infty\}
\end{equation}
equipped with the norm $\|u\|_\infty=\sup_{x\in\R}|u(x)|$.
For given $0<\nu<1$ and $0<\theta<1$, let
\begin{equation}
\label{holder-cont-space}
C^{\nu}_{\rm unif}(\R)=\{u\in C_{\rm unif}^b(\R)\,|\, \sup_{x,y\in\R,x\not =y}\frac{|u(x)-u(y)|}{|x-y|^\nu}<\infty\}
\end{equation}
equipped with the norm $\|u\|_{C^\nu_{\rm unif}}=\sup_{x\in\R}|u(x)|+\sup_{x,y\in\R,x\not =y}\frac{|u(x)-u(y)|}{|x-y|^\nu}$, and
\begin{align*}
C^{\theta}((t_1,t_2),C^\nu_{\rm unif}(\R))=\{ u(\cdot)\in C((t_1,t_2),C^{\nu}_{\rm unif}(\R))\,|\, u(t)\,\, \text{is locally H\"{o}lder continuous with exponent}\,\, \theta\}.
\end{align*}
Among other, the following are proved in \cite{SaSh}.

\medskip

\noindent (i) {\it For any $u_0 \in C_{\rm unif}^{b}(\R)$ with  $u_0 \geq 0$,  there exists $T_{\max}(u_0) \in (0,\infty]$  such that \eqref{IntroEq0-2}  has a unique non-negative classical solution $(u(x,t;u_0),v(x,t;u_0))$ on $[0,T_{\max}(u_0))$ satisfying that $\lim_{t\to 0+}u(\cdot,t;u_0)=u_0$ in the
$C_{\rm unif}^b(\R)$-norm,
$$
u(\cdot,\cdot;u_0) \in C([0, T_{\max}(u_0) ), C_{\rm unif}^b(\R) )\cap C^1((0,T_{\max}(u_0)),C_{\rm unif}^b(\R))
$$
and
$$
u(\cdot,\cdot;u_0),\,\, \partial_{x} u(\cdot,\cdot),\,\, \partial^2_{x x} u(\cdot,\cdot),\,\, \partial_t u(\cdot,\cdot;u_0)\in C^\theta((0,T_{\max}(u_0)),C^\nu_{\rm unif}(\R))
$$
for  $0<\nu\ll 1$.
Moreover, if $T_{\max}(u_0)< \infty,$ then
$\limsup_{t \to T_{\max}(u_0)}   \Big\| u(\cdot,t;u_0) \Big\|_{\infty}  =\infty.$ Furthermore, if $0<\chi<1$, then $T_{\max}=+\infty$ and the solution $(u(\cdot,\cdot,u_{0}),v(\cdot,\cdot,u_{0}))$ is globally bounded.
}

\smallskip

\noindent (ii) {\it Suppose that $u_{0}\in C_{\rm uinf}^{b}(\R)$ with $\inf_{x\in\R}u_{0}(x)>0.$ If $0<\chi<\frac{1}{2}$,
then the unique global classical solution $(u(x,t;u_0),v(x,t;u_0))$ of \eqref{IntroEq0-2} with $u(x,0;u_0)=u_0(x)$  satisfies that
$$
\|u(\cdot,t;u_0)-1\|_{\infty} + \|v(\cdot,t;u_0)-1\|_{\infty}\to 0 \ \text{as} \ t\to \infty.
$$
}

\vspace{-0.1in}\noindent (iii)  {\it  Assume that $0<\chi< \frac{2}{3+\sqrt{2}}$. For any $u_{0}\in C_{\rm unif}^b(\R)$ with $u_0(x)\ge 0$  and ${\rm supp}(u_0)$ being non-empty and compact, there are  $c_{\rm low}^*(u_0)$ and $c_{\rm up}^*(u_0)$ with $0<c_{\rm low}^*(u_0)\le c_{\rm up}^*(u_0)$ such that  the   unique global classical solution $(u(x,t;u_0)$, $v(x,t;u_0))$ of \eqref{IntroEq0-2}  satisfies that
\begin{equation}
\label{spreading-eq1}
\lim_{t\to\infty}\Big[\sup_{|x|\leq ct}|u(x,t;u_0)-1|+\sup_{|x|\leq ct}|v(x,t;u_0)-1|\Big]=0\quad \forall\,\, 0\leq c < c_{\rm low}^{\ast}(u_0),
\end{equation}
\begin{equation}
\label{spreading-eq2}
\lim_{t\to\infty}\Big[\sup_{|x|\geq ct}u(x,t;u_0)+ \sup_{|x|\geq ct}v(x,t;u_0)\Big]=0\quad \forall\,\, c > c_{\rm up}^{\ast}(u_0).
\end{equation}
}

We point out that, considering a chemotaxis model on the whole
space, it is important to study the spatial spreading and
propagating properties of the mobile species in the model. Traveling
wave solutions and spatial spread speeds are among those used to
characterize such properties. There are many studies on traveling
wave solutions of various types of chemotaxis models, see, for
example, \cite{AiHuWa, AiWa, FuMiTs, HoSt, LiLiWa, MaNoSh,
NaPeRy,Wan}, etc. In particular, the reader is referred to the
review paper \cite{Wan}.

The limit properties stated in (iii) in the above  reflect some
spatial spreading feature of the mobile species in
\eqref{IntroEq0-2}. Note that in \cite{NaPeRy}, the authors studied
traveling wave solutions of \eqref{IntroEq0-2} and proved that for
any $0<\chi<1$, there is a $c_*\in [2,2+\frac{\chi}{1-\chi}]$ such
that \eqref{IntroEq0-2} has a traveling wave solution connecting
$(1,1)$ and $(0,0)$ with speed $c_*$  (see \cite[Theorem
1.1]{NaPeRy}).  Besides the above mentioned results, up to our best
knowledge, there is no other existing results on the spatial
spreading and front propagation dynamics of \eqref{IntroEq0-2}.

In the absence of the chemotaxis
(i.e. $\chi=0$), the first equation in  \eqref{IntroEq0-2} becomes the following scalar reaction diffusion equation,
\begin{equation}
\label{fisher-eq}
u_{t}=u_{xx}+ u(1-u),\quad  x\in\R,
\end{equation}
which is referred to as Fisher or KPP equations due to  the pioneering works by Fisher (\cite{Fis}) and Kolmogorov, Petrowsky, Piscunov
(\cite{KPP}) on the spreading properties of \eqref{fisher-eq}. The spatial spreading and  front propagation dynamics of \eqref{fisher-eq} is well understood.
For example,
it follows from the works \cite{Fis}, \cite{KPP}, and \cite{Wei1}  that $c^*_{\rm low}(u_0)$ and $c^*_{\rm up}(u_0)$ can be chosen so that $c^{\ast}_{\rm low}(u_0)=c^{\ast}_{\rm up}(u_0)=2$ for any nonnegative $u_0\in C_{\rm unif}^b(\R^N)$ with ${\rm supp}(u_0)$ being not empty and compact
 ($c^*:=2$ is called the {\it spatial spreading speed} of \eqref{fisher-eq} in literature), and that \eqref{fisher-eq} has traveling wave solutions $u(t,x)=\phi(x-ct)$  connecting $1$ and $0$ (i.e.
$(\phi(-\infty)=1$, $\phi(\infty)=0)$) for all speeds $c\geq c^*$ and has no such traveling wave
solutions of slower speed. Moreover, the stability of traveling wave solutions of \eqref{fisher-eq} connecting $1$ and $0$ has also been proved
(see \cite{Bra}, \cite{Sat}, \cite{Uch}, etc.).
 Since the pioneering works by  Fisher \cite{Fis} and Kolmogorov, Petrowsky,
Piscunov \cite{KPP},  a huge amount research has been carried out toward the front propagation dynamics of
  reaction diffusion equations of the form,
\begin{equation}
\label{general-fisher-eq}
u_t=\Delta u+u f(t,x,u),\quad x\in\R^N,
\end{equation}
where $f(t,x,u)<0$ for $u\gg 1$,  $\partial_u f(t,x,u)<0$ for $u\ge 0$ (see \cite{Berestycki1, BeHaNa1, BeHaNa2, Henri1, Fre, FrGa, LiZh, LiZh1, Nad, NoRuXi, NoXi1, She1, She2, Wei1, Wei2, Zla}, etc.).

  When $\chi> 0$, the study of \eqref{IntroEq0-2} is much difficult because of the lack of comparison principle.
The objective of this paper is to  further investigate the spreading feature of \eqref{IntroEq0-2} and to study
the existence of traveling wave solution of \eqref{IntroEq0-2} connecting $(1,1)$ and $(0,0)$.

A {\it traveling wave solution} of \eqref{IntroEq0-2} connecting $(1,1)$ and $(0,0)$  with speed $c$ is an entire solution $(u(x,t),v(x,t))$ satisfying that $(u(x,t),v(x,t))=(U(x-ct), V(x-ct))$ for
some continuous function $(U(\cdot),V(\cdot))\in C_{\rm unif}^b(\R)\times C_{\rm unif}^b(R)$ with $U(-\infty)=1$ and $U(\infty)=0$.

Observe that \eqref{IntroEq0-2} is equivalent to
\begin{equation}
\label{reduced-main-eq}
\begin{cases}
u_{t}=u_{xx}-\chi u_{x}v_{x} + u(1-\chi v-(1-\chi)u),\quad x\in\R\\
0=v_{xx}-v+u, \quad x\in\R.
\end{cases}
\end{equation}
 Observe also that the function $(0,1)\ni \mu \mapsto \frac{\mu(\mu +\sqrt{1-\mu^2})}{1-\mu^{2}}$ is strictly
  increasing, continuous and satisfies
$$
 \lim_{\mu\to 0^+}\frac{\mu(\mu +\sqrt{1-\mu^2})}{1-\mu^{2}} =0\quad \text{and}\quad \lim_{\mu\to 1^-}\frac{\mu(\mu +\sqrt{1-\mu^2})}{1-\mu^{2}}=\infty.
$$
 Hence the Intermediate Value Theorem implies that that for any $\chi\in (0,1)$, there is a unique $\mu^{*}\in(0, 1)$
 such that
\begin{equation}
\label{mu-star-new-eq}
 \frac{\mu^*(\mu^*
+\sqrt{1-\mu^{*2}}) }{1-\mu^{*2}}=\frac{1-\chi}{\chi}.
\end{equation}
We may denote $\mu^*$ satisfying \eqref{mu-star-new-eq} by $\mu^*(\chi)$ to indicate its dependence on $\chi$.
For given $\chi\in(0,1)$, let
 \begin{equation}
 \label{c-star-eq}
 c^*(\chi)=\mu^*(\chi)+\frac{1}{\mu^*(\chi)}.
 \end{equation}

 We prove the following theorem on traveling wave solutions of
\eqref{IntroEq0-2} or \eqref{reduced-main-eq}.

\medskip

\noindent {\bf Theorem A.} {\it Assume that $0< \chi<\frac{1}{2}$.
Then  for every $c\geq c^*(\chi)$, \eqref{reduced-main-eq} has a
traveling wave solution $(u(x,t),v(x,t))=(U(x-ct),V(x-ct))$ with
speed $c$ and connecting $(1,1)$ and $(0,0)$. Moreover,
\begin{equation}\label{eq-existence-tv-sol-and-spreading-speed}
\lim_{x\to \infty}\frac{U(x)}{e^{-\mu x}}=1,
\end{equation}
where $\mu$ is the only solution of the equation
$c=\mu+\frac{1}{\mu}$ in $(0, 1)$. }

\begin{rk} \label{remark-on-theorem-A}
\begin{itemize}
\item[(i)] By the definition of $c^*(\chi)$, it is easy to see that
\begin{equation}
\label{tv-aux-eq} c^*(\chi)> 2\quad {\rm and}\quad \lim_{\chi\to
0^+}c^*(\chi)=2.
\end{equation}
Hence, as $\chi\to 0+$,  $c^*(\chi)$ converges to the minimal wave
speed (i.e. $2$)  of \eqref{fisher-eq}.

\item[(ii)] Let $0<\chi<\frac{1}{2}$ and
\vspace{-0.1in}\begin{align*}
c_{\min}^*(\chi)=\inf \{c(\chi)\,|\, &\forall\,\,  c\ge c(\chi), \,\, \text{\eqref{reduced-main-eq} has a traveling wave solution}\,\, (u,v)=(U(x-ct),V(x-ct))\\
& {\rm with}\,\, (U(-\infty),V(-\infty))=(1,1),\,\, (U(\infty),V(\infty))=(0,0)\}.
\end{align*}
Theorem A shows that $c_{\rm min}^*(\chi)$ exists and
$c_{\min}^*(\chi)\le c^*(\chi)$. It remains open whether
$c_{\min}^*(\chi)\ge 2$. It also remains open whether
\eqref{reduced-main-eq} has no traveling wave solutions with speed
$c<c_{\rm min}^*(\chi)$ and  connecting $(1,1)$ and $(0,0)$. These
questions reflect the effect of chemotaxis on the wave front
dynamics and are very interesting.

\item[(iii)] The stability and uniqueness of traveling wave solutions of \eqref{reduced-main-eq} connecting
$(1,1)$ and $(0,0)$ is also a very interesting problem. We believe that the limit behavior described in
\eqref{eq-existence-tv-sol-and-spreading-speed} would play a role in the study of this problem.

\item[(iv)] As it is pointed out in the above, the authors in \cite{NaPeRy} proved that for any
$0<\chi<1$,  there is $c_*\in [2, 2+\frac{\chi}{1-\chi}]$ such that
\eqref{reduced-main-eq} has  a  traveling wave solution
 with speed $c_*$ and connecting $(1,1)$ and $(0,0)$. When
 $0<\chi<\frac{1}{2}$, the result in Theorem A  and the result in \cite{NaPeRy} complements
 each other.
 It is interesting to know whether $c^*_{\rm min}(\chi)=c_*$ in
 this case. When $\frac{1}{2}\le \chi<1$, it remains open whether
 \eqref{reduced-main-eq} has traveling wave solutions with sufficiently large speed $c$ and connecting
 $(1,1)$ and $(0,0)$.
 \item[(v)]  Suppose that the logistic source function is replaced by $f(u)=u(a-bu)$ with $a>0$ and $b>0$. For any given $0<\chi<\frac{b}{2}$, let
$\mu^*(\chi)$ be defined by
$$
\mu^*(\chi)=\sup\{\mu\,|\, 0<\mu<\min\{1,\sqrt a\},\,\,\,  \frac{\mu (\mu+\sqrt {1-\mu^2})}{1-\mu^2}\le \frac{b-\chi}{\chi}\}.
$$
Let
$$
c^*(\chi)=\mu^*(\chi)+\frac{a}{\mu^*(\chi)}.
$$
Similarly, we can prove that for any  $c> c^*(\chi)$ ($c$ can also equal $c^*(\chi)$ when $a\ge 1$), \eqref{reduced-main-eq} has a traveling wave solution $(u,v)=(U(x-ct),V(x-ct))$ with speed $c$ connecting the constant solutions $(\frac{a}{b},\frac{a}{b})$ and $(0,0)$. Moreover,
$$
\lim_{\chi\to 0+}c^{*}(\chi)=\begin{cases}
2\sqrt{a}\qquad \qquad  \qquad\ \text{if} \quad 0<a\leq 1\cr
1+a\qquad\qquad\qquad \text{if} \quad a>1.
\end{cases}
\quad \text{and}\quad \lim_{x\to\infty}\frac{U(x)}{e^{-\mu x}}=1,
$$
where $\mu$ is the only solution of the equation $\mu+\frac{a}{\mu}=c$ in the interval $(0\ ,\ \min\{\sqrt{a}, 1\}).$
\end{itemize}
\end{rk}

To state our main results on spreading speeds for \eqref{IntroEq0-2}, we first introduce some standing notations.
Let
$$
C_c^+(\R)=\{u\in C_{\rm unif}^b(\R)\,|\, u(x)\ge 0,\,\, {\rm supp}(u)\,\,\, \text{is non-empty and compact}\}.
$$
Let
$$
C_{-}^*(\chi)=\{c^*_->0\,|\,
 \lim_{t\to\infty} \sup_{|x|\le ct} \big[|u(x,t;u_0)-1|+|v(x,t;u_0)-1|\big]=0\quad \forall\,\, u_0\in C_c^+(\R),\,\, \forall\,  0<c<c_{-}^*\}
 $$
 and
 $$
 C_+^*(\chi)=\{c ^*_+>0\,|\,
 \lim_{t\to\infty}\sup _{|x|\ge ct} \big[ u(x,t;u_0)+v(x,t;u_0)\big]=0\quad \forall\,\, u_0\in C_c^+(\R),\,\, \forall\, c>c_{+}^*\}.
 $$
 Let
 $$
 c_{-}^*(\chi)=\sup\{c\in C_-^*(\chi)\}\quad {\rm and}\quad c_+^*(\chi)=\inf\{c\in C_+^*(\chi)\},
 $$
 where $c_-^*(\chi)=0$ if $C_-^*(\chi)=\emptyset$ and $c_+^*(\chi)=\infty$ if $C_+^*(\chi)=\emptyset$.
  It is clear that
$$
0\le c_-^*(\chi)\le c_+^*(\chi)\le\infty.
$$
 Thanks to the feature of $c_-^*(\chi)$ and $c_+^*(\chi)$, we call the interval $[c_-^*(\chi),c_+^*(\chi)]$
the {\it spreading speed interval} of \eqref{IntroEq0-2}.
We prove the following theorem on the upper and lower bounds of the spreading speed interval $[c_-^*(\chi),c_+^*(\chi)]$ of \eqref{IntroEq0-2}.

\medskip
\noindent {\bf Theorem B.}
{\it
\begin{itemize}
\item[(i)] If $0<\chi<1$, then
 \begin{equation}
 \label{spreading-eq3}
  c_+^*(\chi)\le\min\{ 2+\frac{\chi}{1-\chi},  c^*(\chi)\}.
\end{equation}
where $c^*(\chi)$ is as in \eqref{c-star-eq}.

\item[(ii)] If $0<\chi<\frac{2}{3+\sqrt 2}$, then
\begin{equation}
\label{spreading-eq4}
 c_-^*(\chi)\ge 2\sqrt{1-\frac{\chi}{1-\chi}}-\frac{\chi}{1-\chi}>0.
\end{equation}
\end{itemize}
}

\begin{rk}
\label{spreading-rk}
\begin{itemize}

\item[(i)] Observe that $2\leq \lim_{\chi\to 0+} c_-^*(\chi)\leq \lim_{\chi\to 0+}c_+^*(\chi)\le 2$.  Hence the spreading speed interval
$[c_-^*(\chi),c_+^*(\chi)]$
converges to the single point $\{2\}$ as $\chi\to 0+$, which is the spreading speed of  \eqref{fisher-eq}.

\item[(ii)]  For any given $u_0\in C_c^+(\R)$, $c_{\rm low}^*(u_0)$ and $c_{\rm up}^*(u_0)$ can be chosen so that
 $ c_-^*(\chi)\le c_{\rm low}^*(u_0)\le c_{\rm up}^*(u_0)\le c_+^*(\chi)$.

\item[(iii)]
When the source function in \eqref{IntroEq0-2} is replaced by  $f(u)=u(a-bu)$, similarly, we can prove that
 if $0<\chi<b$, then
 \begin{equation}
 \label{spreading-eq5}
 0\le c_-^*(\chi)\le c_+^*(\chi)<\infty,
\end{equation}

and if $0<\chi<\frac{2b}{3+\sqrt {a+1}}$, then
\begin{equation}
\label{spreading-eq6}
0<2\sqrt{a-\frac{a\chi}{b-\chi}}-\frac{a\chi}{b-\chi}\le c_-^*(\chi)\le c_+^*(\chi)\le  2\sqrt a+\frac{a\chi}{b-\chi}<\infty,
\end{equation}

where $c_-^*(\chi)$ and $c_+^*(\chi)$ are such that
$$ \lim_{t\to\infty} \sup_{|x|\le ct} \big[|u(x,t;u_0)-\frac{a}{b}|+|v(x,t;u_0)-\frac{a}{b}|\big]=0\quad \forall u_0\in C_c^+(\R),\,\,
0<c<c_-^*(\chi)
$$
and
$$
 \lim_{t\to\infty}\sup _{|x|\ge ct} \big[ u(x,t;u_0)+v(x,t;u_0)\big]=0\quad \forall\,\, u_0\in C_c^+(\R),\,\, c> c_+^*(\chi).
$$

\item[(iv)] Regarding the spatial spreading speeds of \eqref{IntroEq0-2}, there are still many interesting problems to be studied.
 For example, whether $c_-^*(\chi)=c_+^*(\chi)$;  whether $c_+^*(\chi)=c^*(\chi)$ for $0<\chi<\frac{1}{2}$;
   what is the relation between $c^*_-(\chi)$, $c^*_+(\chi)$ and $2$ for $0<\chi<1$. These questions are
    important in the understanding of the spreading feature of \eqref{IntroEq0-2} because they are related to the issue whether the chemotaxis speeds up or slows down the spreading of the species. 
\end{itemize}
\end{rk}

We now consider the extensions of Theorems A and B for \eqref{IntroEq0-2} to \eqref{IntroEq0-2-new}. We have the following theorems.

\medskip

\noindent {\bf Theorem C.} {\it Assume that $0< \chi<\frac{1}{2}$.
Let $c^*(\chi)$ be as in Theorem A. Then for any $c\ge c^*(\chi)$
and $\xi\in S^{N-1}$, \eqref{IntroEq0-2-new} has a traveling wave
solution  which connects $(1,1)$ and $(0,0)$ and propagates in the
direction of $\xi\in S^{N-1}$ with speed $c$ } (see Section 5 for
the detail).

\medskip

\noindent {\bf Theorem D.}
{\it Consider \eqref{IntroEq0-2-new}. Let $[c_-^*(\chi),c_+^*(\chi)]$ be the spreading speed interval of
\eqref{IntroEq0-2-new} (see Section 5 for the detail).
\begin{itemize}
\item[(i)] If $0<\chi<1$, then
 \begin{equation}
 \label{spreading-general-eq3}
  0\le c_-^*(\chi)\le c_+^*(\chi) \leq \min\{2+\frac{{ \sqrt{N}} \chi  }{1-\chi},  \frac{1}{\mu_N^*}+\mu_N^*\},
\end{equation}
where $\mu_N^*\in (0,\frac{1}{\sqrt N})$ solves the equation
\begin{equation}
\frac{2^{N}\sqrt{N}\mu_N^*(\mu_N^* +\sqrt{1-N\mu_N^{*2}}) }{1-N\mu_N^{*2}}=\frac{1-\chi}{\chi} .
\end{equation}

\item[(ii)]  If $0<\chi<\frac{2}{3+\sqrt{N+1} }$, then
\begin{equation}
\label{spreading-general-eq4}
 c_-^*(\chi)\ge  2\sqrt{1-\frac{\chi}{1-\chi}}-\frac{\chi\sqrt{N}}{1-\chi}>0.
\end{equation}
\end{itemize}
}

\medskip

  Because of the lack of comparison principle,  the proofs of Theorems A - D are highly non trivial. Our approach to prove  Theorem A is based on the construction of a bounded convex non-empty subset of $C_{\rm unif}^{b}(\R)$, called $\mathcal{E}_{\mu}$(see \eqref{definition-E-mu}), and a continuous and compact function  $U : \mathcal{E}_{\mu}\to \mathcal{E}_{\mu}$. Any fixed point of this function, whose existence is guaranteed by the Schauder's fixed theorem,  becomes a traveling solution of  \eqref{IntroEq0-2}. The construction of the set $\mathcal{E}_{\mu}$ itself is also based on the construction of two special functions. These two special functions are sub-solution and sup-solution of a collection of parabolic equations. At each $u\in\mathcal{E}_{\mu}$ we shall first associate a function which is the solution of a certain parabolic equation, and next define $U(\cdot,u)$ to be the pointwise limit as $t$ goes to infinity of the previous function.  One important ingredient in the proof of
Theorem B is to prove that  for  any $u_{0}\in C^{+}_{c}(\R)$, there is $M>0$ such that
$$
0\le  u(x,t;u_0)\le  M e^{-\mu^*(\chi) (|x|-c_{\mu^*(\chi)} t)},
$$
where $(u(x,t;u_0),v(x,t;u_0))$ is the solution of \eqref{IntroEq0-2}  with $u(x,0;u_0)=u_0(x)$.
To do so, for given $u_0\in C_c^+(\R)$ and $T>0$, we also construct a bounded convex non-empty subset $\mathcal{E}_{\mu}^T(u_0)$
of $C_{\rm unif}^b(\R\times [0,T])$ and a continuous and compact function  $\bar U : \mathcal{E}_{\mu^*(\chi)}^T(u_0)\to \mathcal{E}_{\mu^*(\chi)}^T(u_0)$.
Then we prove $u(\cdot,\cdot;u_0)|_{\R\times[0,T]}$ is a fixed point of $\bar U$. We use the ideas in the proofs of Theorems A and B and some results
in Theorems A and B to prove Theorems C and D.

\smallskip

The rest of this paper is organized as follows. Section 2 is to establish the tools that will be needed in the proof of our main results. It is here that we define the two special functions, which are sub-solution and sup-solution of a collection of parabolic equations,  and the non-empty bounded and convex subset $\mathcal{E}_{\mu}$. In sections 3 and 4, we
prove the  main results on  the existence of traveling wave solutions and on the spreading speeds for \eqref{IntroEq0-2}, respectively.
We give the idea of proofs of Theorems C and D in section 5.

\section{Super- and sub-solutions}

In this section, we will construct super- and sub-solutions of some related equations of \eqref{reduced-main-eq}, which will be used to prove the existence of traveling wave solutions of \eqref{reduced-main-eq} in next section.

 Observe that, if $(u(x,t),v(x,t))=(U(x-ct),V(x-ct))$ is a traveling wave solution of \eqref{reduced-main-eq} connecting $(1,1)$ and $(0,0)$ with speed $c$, then $(u,v)=(U(x),V(x))$ is a stationary solution of
\begin{equation}
\label{stationary-eq}
\begin{cases}
u_{t}=u_{xx}+c u_x-\chi u_{x}v_{x} + u(1-\chi v-(1-\chi)u),\quad x\in\R\\
0=v_{xx}-v+u, \quad x\in\R
\end{cases}
\end{equation}
connecting $(1,1)$ and $(0,0)$. For given $c$, to show the existence of a traveling wave solution of \eqref{reduced-main-eq} connecting $(1,1)$ and
 $(0,0)$ is then equivalent to
show the existence of a stationary solution connecting $(1,1)$ and $(0,0)$.  Throughout this section, we assume that
$0<\chi<1$, unless specified otherwise.

For every $0<\mu<1$ and $x\in \R$ define
$$\varphi_{\mu}(x)=e^{-\mu x}\quad {\rm and}\quad c_{\mu}=\mu+\frac{1}{\mu}.
 $$
 Note that for every fixed $0<\mu<1$, the function $\varphi_{\mu}$ is decreasing, infinitely many differentiable, and satisfied
 \begin{equation}\label{Eq1 of varphi}
\varphi_{\mu}''(x)+c_{\mu}\varphi_{\mu}'(x)+\varphi(x)=0 \quad\forall\ x\in\R
\end{equation}
and
\begin{equation}\label{Eq2 of varphi}
\frac{1}{1-\mu^2}\varphi_{\mu}''(x) - \frac{1}{1-\mu^2}\varphi_{\mu}(x)=-\varphi_{\mu}(x)\quad \forall\,\, x\in\R.
\end{equation}

For every $\mu\in (0, 1)$ define
\begin{equation}
U_{\mu}^{+}(x)=\min\{\frac{1}{1-\chi}, \varphi_{\mu}(x)\}=\begin{cases}
\frac{1}{1-\chi} \ \quad \text{if }\ x\leq \frac{\ln(1-\chi)}{\mu}\\
e^{-\mu x} \quad \ \text{if}\ x\geq \frac{\ln(1-\chi)}{\mu}.
\end{cases}
\end{equation}
and
\begin{equation}V_{\mu}^{+}(x)=\min\{\frac{1}{1-\chi}, \,\  \frac{1}{1-\mu^{2}}\varphi_{\mu}(x)\}.
\end{equation}
Since $\varphi_{\mu}$ is decreasing, then the functions $U^{+}_{\mu}$ and $V_{\mu}^{+}$ are both non-increasing. Furthermore, the functions $U^{+}_{\mu}$ and $V_{\mu}^{+}$ belong to $C^{\delta}_{\rm unif}(\R)$ for every $0\leq \delta< 1$ and  $0< \mu<1$.

Let $0< \mu<1$ be fixed. Next, let $\mu<\tilde{\mu}<\min\{1,2\mu\}$ and $d>1$. The function $\varphi_{\mu}-d\varphi_{\tilde{\mu}}$ achieved its maximum value at $\bar{a}_{\mu,\tilde{\mu},d}:=\frac{\ln(d\tilde{\mu})-\ln(\mu)}{\tilde{\mu}-\mu}$ and takes the value zero at $\underline{a}_{\mu,\tilde{\mu},d}:= \frac{\ln(d)}{\tilde{\mu}-\mu}$.
Define
\begin{equation}
U_{\mu}^{-}(x):= \max\{ 0, \varphi_{\mu}(x)-d\varphi_{\tilde{\mu}}(x)\}=\begin{cases}
0\qquad \qquad \qquad \quad \text{if}\ \ x\leq \underline{a}_{\mu,\tilde{\mu},d}\\
\varphi_{\mu}(x)-d\varphi_{\tilde{\mu}}(x)\quad \text{if}\ x\geq \underline{a}_{\mu,\tilde{\mu},d}.
\end{cases}
\end{equation}
Clearly, $0\leq U_{\mu}^{-}\leq U^{+}_{\mu}\leq \frac{1}{1-\chi}$ and $U_{\mu}^{-}\in C^{\delta}_{\rm unif}(\R)$ for every $0\leq \delta< 1$.

\medskip

Let us consider the set $\mathcal{E}_{\mu}$ defined by
\begin{equation}\label{definition-E-mu}
\mathcal{E}_{\mu}=\{u\in C^{b}_{\rm unif}(\R) \,|\, U_{\mu}^{-}\leq u\leq U_{\mu}^{+}\}
\end{equation}
for every $0<\mu<1$. It should be noted that $U_{\mu}^{-}$ and $\mathcal{E}_{\mu}$ all depend on $\tilde{\mu}$ and $d$. Later on, we shall provide more information on how to choose $d$ and $\tilde{\mu}$ whenever $\mu$ is given.

For every $u\in C_{\rm unif}^b(\R)$, consider
\begin{equation}\label{ODE2}
U_{t}=U_{xx}+(c_{\mu}-\chi V'(x;u))U_{x}+(1-\chi V(x;u)-(1-\chi)U)U, \quad x\in \R,\ t>0,
\end{equation}
where
\begin{equation}\label{Inverse of u}
V(x;u)=\int_{0}^{\infty}\int_{\R}\frac{e^{-s}}{\sqrt{4\pi s}}e^{-\frac{|x-z|^{2}}{4s}}u(z)dzds.
\end{equation}
It is well known that the function $V(x;u)$
is the solution of the second equation of \eqref{reduced-main-eq} in $C^{b}_{\rm unif}(\R)$ with given $u\in C_{\rm unif}^b(\R)$.

For  given open intervals $D\subset \R$ and $I\subset \R$, a function $U(\cdot,\cdot)\in C^{2,1}(D\times I,\R)$ is called a {\it super-solution} or {\it sub-solution} of \eqref{ODE2} on $D\times I$ if
$$U_{t}\ge U_{xx}+(c_{\mu}-\chi V'(x;u))U_{x}+(1-\chi V(x;u)-(1-\chi)U)U \quad {\rm for}\,\, x\in D,\,\,\, t\in I
$$
or
$$
U_{t}\le U_{xx}+(c_{\mu}-\chi V'(x;u))U_{x}+(1-\chi V(x;u)-(1-\chi)U)U \quad {\rm for}\,\, x\in D,\,\,\, t\in I.
$$
\begin{tm}
\label{super-sub-solu-thm}
Suppose that $0<\chi<\frac{1}{2}$ and  $0<\mu<1$ satisfy
\begin{equation}\label{Eq01_Th1}
 \frac{\mu(\mu+\sqrt{1-\mu^2})}{1-\mu^2} \leq \frac{1- \chi}{\chi}.
\end{equation}

 Then for every $u\in \mathcal{E}_{\mu}$, the following hold.
\begin{itemize}
\item[(1)] $U(x,t)=\frac{1}{1-\chi}$ and $U(x,t)=\varphi_\mu(x)$ are supper-solutions of \eqref{ODE2} on $\R\times\R$.

\item[(2)] There is $d_0>0$ such that $U(x,t)=U_\mu^-(x)$ is a sub-solution of \eqref{ODE2} on
$(\underline{a}_{\mu,\tilde{\mu},d},\infty)\times \R$ for all $d\ge d_0$ and $\mu< \tilde{\mu}<\min\{1,2\mu,\mu+ \frac{1}{\mu+\sqrt{1-\mu^2}}\}$.
Moreover, $U(x,t)=U_\mu^-(x_\delta)$ is a sub-solution of \eqref{ODE2} on $\R\times \R$ for $0<\delta\ll 1$,
where $x_\delta=\underline{a}_{\mu,\tilde{\mu},d}+\delta$.
\end{itemize}
\end{tm}

To prove Theorem \ref{super-sub-solu-thm}, we first establish some estimates on
$V(\cdot;u)$ and $V^{'}(\cdot;u)$.

It was established in \cite{SaSh} that
\begin{equation}\label{Estimates on Inverse of u}
\max\{\|V(\cdot;u)\|_{\infty}, \ \|V'(\cdot;u)\|_{\infty} \}\leq \|u\|_{\infty}\quad \forall\ u\in C^{b}_{\rm unif}(\R).
\end{equation}
Furthermore, let
$$
C_{\rm unif}^{2,b}(\R)=\{u\in C_{\rm unif}^b(\R)\,|\, u^{'}(\cdot),\, u^{''}(\cdot)\in C_{\rm unif}^b(\R)\}.
$$
For every $u\in C_{\rm unif}^{b}(\R)$, $u\geq 0,$ we have that  $V(\cdot;u)\in C^{2,b}_{\rm unif}(\R)$ with $V(\cdot;u)\geq 0$ and
$$\| V''(\cdot;u)\|_{\infty}=\|V(\cdot;u)-u\|_{\infty}\leq \max\Big\{\|V(\cdot;u)\|_{\infty},\|u\|_{\infty}\Big\}.
$$
Combining this with inequality \eqref{Estimates on Inverse of u}, we obtain that
\begin{equation}\label{Estimates on Inverse of V}
\max\{\|V(\cdot;u)\|_{\infty}, \ \|V'(\cdot;u)\|_{\infty}, \|V''(\cdot;u)\|_{\infty}\}\leq \|u\|_{\infty}\quad \forall\ u\in \mathcal{E}_{\mu}.
\end{equation}

The next Lemma provide a pointwise estimate  for $|V(\cdot;u)|$ whenever $u\in \mathcal{E}_{\mu}.$

\begin{lem}\label{Mainlem2}
For every $0<\mu<1$ and $u\in \mathcal{E}_{\mu}$, let $V(\cdot;u)$ be defined as in \eqref{Inverse of u}, then
\begin{equation}\label{Eq_MainLem2}
0\leq  V(\cdot;u)\leq V^{+}_{\mu}(\cdot).
\end{equation}
\end{lem}
\begin{proof} For every  $u\in \mathcal{E}_{\mu}$, since $ 0 \leq U^{-}_{\mu}\leq u\leq U^{+}_{\mu}$ then
$$ 0\leq V(\cdot;U^{-}_{\mu})\leq V(\cdot;u)\leq V(\cdot;U^{+}_{\mu}).$$ Hence it is enough to prove that $V(\cdot;U_{\mu}^{+})\leq V^{+}_{\mu}(\cdot)$.  For every $x\in \R$, $0<\mu<1$,  we have that
\begin{eqnarray}\label{Eq011}
\int_{0}^{\infty}\Big(\int_{\R}\frac{e^{-s}e^{-\frac{|x-z|^2}{4s}}\varphi_{\mu}( z)}{\sqrt{4\pi s}}dz\Big)ds & = &\frac{1}{\sqrt{\pi}}\int_{0}^{\infty}e^{-s}\Big(\int_{\R}e^{- z^2}e^{-\mu (x-2\sqrt{s}z)}dz\Big)ds\nonumber\\
&=& \frac{e^{-\mu x}}{\sqrt{\pi}}\int_{0}^{\infty}e^{-s}\Big(\int_{\R}e^{- |z-\mu\sqrt{s}|^{2}}e^{\mu^{2} s}dz\Big)ds\nonumber\\
&=& \frac{e^{-\mu x}}{\sqrt{\pi}}\int_{0}^{\infty}e^{-(1-\mu^2)s}\Big({\int_{\R}e^{- |z-\mu\sqrt{s}|^{2}}dz}\Big)ds\nonumber\\
&=& e^{-\mu x}\int_{0}^{\infty}e^{-(1-\mu^2)s}ds\nonumber\\
&=& \frac{\varphi_{\mu}( x)}{1-\mu^{2}}.
\end{eqnarray}
Thus, we have
\begin{eqnarray*}
V(x;U^{+}_{\mu})&= & \int_{0}^{\infty}\Big(\int_{\R}\frac{e^{-s}e^{-\frac{|x-z|^2}{4s}}U^{+}_{\mu}( z)}{\sqrt{4\pi s}}dz\Big)ds\nonumber\\
&= & \int_{0}^{\infty}\Big(\int_{\R}\frac{e^{-s}e^{-\frac{|x-z|^2}{4s}}}{\sqrt{4\pi s}}\min\{ \frac{1}{1-\chi} \ ,\ \varphi_{\mu}(z) \}dz\Big)ds\nonumber\\
&\leq & \min\Big\{ \frac{1}{1-\chi}\underbrace{\int_{0}^{\infty}\int_{\R}\frac{e^{-s}e^{-\frac{|x-z|^2}{4s}}}{\sqrt{4\pi s}}dzds}_{=1}\ ,\ \int_{0}^{\infty}\Big(\int_{\R}\frac{e^{-s}e^{-\frac{|x-z|^2}{4s}}}{\sqrt{4\pi s}} \varphi_{\mu}(z)dz\Big)ds\Big\}\nonumber\\
&=& V^{+}_{\mu}(x).
\end{eqnarray*}
\end{proof}

Next, we present a pointwise  estimate  for $|V'(\cdot;u)|$ whenever $u\in \mathcal{E}_{\mu}.$

\begin{lem}\label{Mainlem3}
Let $u\in C^{b}_{\rm unif}(\R)$ and $V(\cdot;u)\in C^{2,b}_{\rm unif}(\R)$ be the corresponding function satisfying the second equation of \eqref{reduced-main-eq}. Then
\begin{equation}\label{Eq_Mainlem01}
|V'(x;u)|\leq \frac{\mu+\sqrt{1-\mu^2}}{1-\mu^2}\varphi_{\mu}(x)
\end{equation}
for every $x \in\R$ and every $u\in\mathcal{E}_{u}$.
\end{lem}

\begin{proof} Let $u\in\mathcal{E}_{\mu}$ and fix any $x\in \R$.
\begin{align}\label{Eq_Mainlem0001}
V'(x;u)=\int_{0}^{\infty}\int_{\R}\frac{(z-x)e^{-s}}{2s\sqrt{4\pi s }}e^{-\frac{|z-x|^{2}}{4s}}u(z)dzds=\frac{1}{\sqrt{\pi}}\int_{0}^{\infty}\int_{\R}\frac{ze^{-s}}{\sqrt{ s }}e^{-z^{2}}u(x+2\sqrt{s}z)dzds.
\end{align}

Observe that
\begin{eqnarray}\label{Eq_Mainlem0002}
\frac{1}{\sqrt\pi}\int_{0}^{\infty}\int_{\R}\frac{|z|}{\sqrt{s}}e^{-s}e^{-|z|^2}\varphi_{\mu}(x+2\sqrt{s}z)dzds &\leq &  \frac{\varphi_{\mu}(x)}{\sqrt{\pi}}\int_{0}^{\infty}\frac{e^{-(1-\mu^2)s}}{\sqrt{s}}\Big(\int_{\R}|z|e^{-|z-\mu\sqrt{s}|^2}dz\Big)ds \nonumber \\
&=&  \frac{\varphi_{\mu}(x)}{\sqrt{\pi}}\int_{0}^{\infty}\frac{e^{-(1-\mu^2)s}}{\sqrt{s}}\Big(\int_{\R}|z+\mu\sqrt{s}|e^{-|z|^2}dz\Big)ds \nonumber\\
&\leq &  \frac{\varphi_{\mu}(x)}{\sqrt{\pi}}\int_{0}^{\infty}\frac{e^{-(1-\mu^2)s}}{\sqrt{s}}\Big(\int_{\R}(|z|+\mu\sqrt{s})e^{-|z|^2}dz\Big)ds \nonumber\\
&= &  \frac{\varphi_{\mu}(x)}{\sqrt{\pi}}\int_{0}^{\infty}\frac{(1+\mu\sqrt{\pi s})e^{-(1-\mu^2)s}}{\sqrt{s}}ds \nonumber\\
&= &(\frac{1}{\sqrt{1-\mu^2}} +\frac{\mu}{1-\mu^2})\varphi_{\mu}(x).
\end{eqnarray}
Since $u\leq \varphi_{\mu}$, \eqref{Eq_Mainlem01} follows from \eqref{Eq_Mainlem0001} and \eqref{Eq_Mainlem0002}.
The Lemma is thus proved.
\end{proof}

\begin{proof}[Proof of Theorem \ref{super-sub-solu-thm}]
 For every $U\in C^{2,1}(\R\times\R_{+})$, let \begin{equation}\label{mathcal L}
\mathcal{L}U=U_{xx}+(c_{\mu}-\chi V'(\cdot;u))U_{x}+(1-\chi V(\cdot;u)-(1-\chi)U)U .
\end{equation}
(1)  First, we have that
\begin{align*}
\mathcal{L}( \frac{1}{1-\chi})&= (1-\chi V(\cdot;u)-1)\frac{1}{1-\chi}\\
&= -\frac{\chi}{1-\chi} V(\cdot;u)\\
&\le 0.
\end{align*}
Hence $U(x,t)=\frac{1}{1-\chi}$ is a super-solution of \eqref{ODE2} on $\R\times\R$.

Next, it follows from Lemma \ref{Mainlem3} and \eqref{Eq01_Th1} that
\begin{eqnarray}\label{H}
 \mathcal{L}(\varphi_{\mu})& =& \varphi''_{\mu}(x)+(c_{\mu}-\chi V'(\cdot;u))\varphi_{\mu}'(x)+(1-\chi V(\cdot;u)-(1-\chi)\varphi_{\mu})\varphi_{\mu}\nonumber\\
 & =& \underbrace{(\varphi''_{\mu}+c_{\mu}\varphi'_{\mu}+\varphi_{\mu})}_{=0} +(\mu\chi V'(\cdot;u)-\chi V(\cdot;u)-(1-\chi)\varphi_{\mu} )\varphi_{\mu}\nonumber\\
  & =& (\mu\chi V'(\cdot;u)-\chi V(\cdot;u)-(1-\chi)\varphi_{\mu} )\varphi_{\mu}\nonumber\\
   & \leq & \chi\Big(  \frac{\mu(\mu+\sqrt{1-\mu^2})}{1-\mu^2}-\frac{(1-\chi)}{\chi}\Big)\varphi_{\mu}^{2} \nonumber\\
   &\leq & 0.
\end{eqnarray}
Hence $U(x,t)=\varphi_\mu(x)$ is also a super-solution of \eqref{ODE2} on $\R\times\R$.

(2)  Let $O=(\underline{a}_{\mu,\tilde{\mu},d},\infty)$. Then for $x\in O$, $U_\mu^-(x)>0$.
 For $x\in O$, it follows from inequality \eqref{Eq_Mainlem01} that
\begin{eqnarray*}
\mathcal{L}U_{\mu}^{-}
& =& \mu^2\varphi_{\mu}-\tilde{\mu}^2d\varphi_{\tilde{\mu}} +(c_{\mu}-\chi V'(\cdot;u))(-\mu\varphi_{\mu}+d\tilde{\mu}\varphi_{\tilde{\mu}})+(1-\chi V(\cdot;u)-(1-\chi) U_{\mu}^{-})U_{\mu}^{-} \nonumber\\
& =&\underbrace{(\mu^{2}-\mu c_{\mu}+1)}_{=0}\varphi_{\mu} +d\underbrace{(\tilde{\mu}c_{\mu}-\tilde{\mu}^{2}-1)}_{=A_{0}}\varphi_{\tilde{\mu}} -\chi V'(\cdot;u)(-\mu\varphi_{\mu}+d\tilde{\mu}\varphi_{\tilde{\mu}})   -(\chi V +(1-\chi)U_{\mu}^{-})U_{\mu}^{-}\nonumber\\
& \geq  & dA_{0}\varphi_{\tilde{\mu}} -\chi|V'(\cdot;u)|(\mu\varphi_{\mu}+d\tilde{\mu}\varphi_{\tilde{\mu}}) -\chi V^{+}_{\mu}U_{\mu}^{-}-(1-\chi)[U^{-}_{\mu}]^2\nonumber\\
& \geq & dA_{0}\varphi_{\tilde{\mu}} -\chi \frac{(\mu+\sqrt{1-\mu^2})}{1-\mu^2}\Big(\mu\varphi_{\mu}+d\tilde{\mu}\varphi_{\tilde{\mu}}\Big)\varphi_{\mu} -\chi V^{+}_{\mu}U_{\mu}^{-}-(1-\chi)[U^{-}_{\mu}]^2\nonumber\\
& \geq & dA_{0}\varphi_{\tilde{\mu}} -\chi \frac{(\mu+\sqrt{1-\mu^{2}})}{1-\mu^2}\Big(\mu\varphi_{\mu}+d\tilde{\mu}\varphi_{\tilde{\mu}}\Big)\varphi_{\mu} -\frac{\chi}{1-\mu^2}\varphi_{\mu}U^{-}_{\mu}-(1-\chi)[U^{-}_{\mu}]^2\nonumber\\
& = & dA_{0}\varphi_{\tilde{\mu}} -\underbrace{(\chi \frac{\mu(\mu+\sqrt{1-\mu^{2}})}{1-\mu^2}+\frac{\chi}{1-\mu^2}+1-\chi)}_{=A_{1}}\varphi_{\mu}^2 \nonumber\\
& & +d\Big(2(1-\chi)-\chi\frac{\tilde{\mu}(\mu+\sqrt{1-\mu^2})}{1-\mu^2}+\frac{\chi}{1-\mu^2}\Big)\varphi_{\mu}\varphi_{\tilde{\mu}}-d^2(1-\chi)\varphi_{\tilde{\mu}}^2.
\end{eqnarray*}
Note that  $U_{\mu}^{-}(x)>0$ is equivalent to $\varphi_{\mu}(x)>d\varphi_{\tilde{\mu}}(x)$, which is again equivalent to
$$
d(1-\chi)\varphi_{\mu}(x)\varphi_{\tilde{\mu}}(x)>d^{2}(1-\chi)\varphi^2_{\tilde{\mu}}(x).
$$
Since $A_{1}>0$, thus for $x\in O$, we have
\begin{eqnarray*}
\mathcal{L}U_{\mu}^{-}(x) & \geq &  dA_{0}\varphi_{\tilde{\mu}}(x) -A_{1}\varphi_{\mu}^2(x) +d\underbrace{\Big((1-\chi)-\chi\frac{\tilde{\mu}(\mu+\sqrt{1-\mu^2})}{1-\mu^2}+\frac{\chi}{1-\mu^2}\Big)}_{A_{2}}\varphi_{\mu}(x)\varphi_{\tilde{\mu}}(x)\nonumber\\
& =& A_{1}\Big(\frac{dA_{0}}{A_{1}}e^{(2\mu-\tilde{\mu})x}-1\Big)\varphi_{\mu}^{2}(x) +dA_{2}\varphi_{\mu}(x)\varphi_{\tilde{\mu}}(x).
\end{eqnarray*}
Note also that, by \eqref{Eq01_Th1},
\begin{eqnarray}\label{Eq1 of Th2}
A_{2}&=&\chi\Big(\frac{1-\chi}{\chi}- \frac{\mu(\mu+\sqrt{1-\mu^2})}{1-\mu^2}\Big)+ \frac{\chi}{1-\mu^2}\Big(1-(\tilde{\mu}-\mu)(\mu +\sqrt{1-\mu^2}) \Big) \nonumber\\
&\geq & \frac{\chi}{1-\mu^2}\Big(1-(\tilde{\mu}-\mu)(\mu +\sqrt{1-\mu^2}) \Big)\ge 0,
\end{eqnarray}
whenever $\tilde{\mu}\leq \mu+ \frac{1}{\mu+\sqrt{1-\mu^2}}$.
Observe that
$$
A_{0}=\frac{(\tilde{\mu}-\mu)(1-\mu\tilde{\mu})}{\mu}>0.
$$
 Furthermore, we have that $U_{\mu}^{-}(x)>0$ implies that $x>0$ for $d\geq 1$. Thus, for every $ d\geq d_{0}:= \max\{1, \frac{A_{1}}{A_{0}}\}$, we have that
\begin{equation}\label{E1}
\mathcal{L}U_{\mu}^{-}(x) > 0
\end{equation}
whenever $x\in O$ and $\tilde{\mu}\leq \min\{2\mu,\mu+ \frac{1}{\mu+\sqrt{1-\mu^2}}\}$. Hence $U(x,t)=U_\mu^-(x)$ is a sub-solution of \eqref{ODE2} on $(\underline{a}_{\mu,\tilde{\mu},d},\infty)\times\R$.

Note that for $0<\delta\ll 1$,
\begin{eqnarray*}
(1-\chi V(x_\delta;u)-(1-\chi)U_\mu^-(x_\delta))U_\mu^-(x_\delta)& \geq & (1-\frac{\chi}{1-\chi}-(1-\chi)U_\mu^-(x_\delta))U_\mu^-(x_\delta) \nonumber\\
&> & 0\quad\forall\,\, x\in\R,
\end{eqnarray*}
whenever $0<\chi<\frac{1}{2}$, where $x_\delta=\underline{a}_{\mu,\tilde{\mu},d}+\delta$. This implies that $U(x,t)=U_\mu^-(x_\delta)$ is
a sub-solution of \eqref{ODE2} on $\R\times\R$.
\end{proof}

\section{Traveling wave solutions}

In this section, we investigate the existence of traveling wave solutions of \eqref{reduced-main-eq} connecting $(1,1)$ and
$(0,0)$ and prove Theorem A. We first prove  the following theorem and then prove Theorem A.

\begin{tm}
\label{existence-tv-thm}
Suppose that  $0<\mu<1 $ and $0<\chi<\frac{1}{2}$ satisfy  \eqref{Eq01_Th1}. Let $c_\mu=\mu+\frac{1}{\mu}$.
Then \eqref{reduced-main-eq} has a traveling wave solution
$(u(x,t),v(x,t))=(U(x-c_\mu t),V(x-c_\mu t))$ satisfying
$$
\lim_{x\to-\infty}U(x)=1,\quad \lim_{x\to\infty}\frac{U(x)}{e^{-\mu x}}=1.
$$

\end{tm}

Our key idea to prove the above theorem is to prove that, for any $\mu>0$ and $0<\chi<\frac{1}{2}$ satisfying \eqref{Eq01_Th1}, there is $u^*(\cdot)\in\mathcal{E}_\mu$ such that
$U=u^*(\cdot)$ is a stationary solution of \eqref{ODE2} and $u^*(-\infty)=1$ and $u^*(\infty)=0$, which implies that
$(u(x,t),v(x,t))=(u^*(x-c_\mu t),V(x-c_\mu t;u^*))$ is a traveling wave solution of \eqref{IntroEq0-2} connecting $(1,1)$ and $(0,0)$.

In order to prove Theorem \ref{existence-tv-thm}, we first prove some lemmas. Fix $u\in\mathcal{E}_\mu$. For given $u_0\in C_{\rm unif}^b(\R)$, let
  $U(x,t;u_0)$ be the solution of \eqref{ODE2} with
$U(x,0;u_0)=u_0(x)$. By the arguments in the proof of Theorem 1.1 and Theorem 1.5 in \cite{SaSh}, we have $U(x,t;U_\mu^+)$ exists for all $t>0$ and
${ U(\cdot,\cdot;U_\mu^+)}\in C([0,\infty),C^{b}_{\rm unif}(\R))\cap C^{1}((0\ ,\ \infty),C^{b}_{\rm unif}(\R))\cap C^{2,1}(\R\times(0,\ \infty))$ satisfying
\begin{equation}
U(\cdot,\cdot;U_\mu^+), U_{x}(\cdot,\cdot;U_\mu^+),U_{xx}(\cdot,t;U_\mu^+),U_{t}(\cdot,\cdot;U_\mu^+)\in  C^{\theta}((0, \infty),C_{\rm unif}^{\nu}(\R))
\end{equation}
for $0<\theta, \nu \ll 1$.

\begin{lem} \label{lm1}
Assume that $0<\mu,\chi<1$ satisfy
\eqref{Eq01_Th1}.
 Then for every $u\in \mathcal{E}_{\mu}$, the following hold.
\begin{description}
\item[(i)] $0\leq U(\cdot,t;U_\mu^+)\leq U_{\mu}^{+}(\cdot)$ for every $t\geq 0.$
\item[(ii)] $U(\cdot,t_{2};U_\mu^+)\leq U(\cdot,t_{1};U_\mu^+) $ for every $0\leq t_{1}\leq t_{2}$
\end{description}
\end{lem}

\begin{proof}
(i)   Note that $U^{+}_{\mu}(\cdot)\leq  \frac{1}{1-\chi} $. Then by
comparison principle for parabolic equations and Theorem \ref{super-sub-solu-thm}(1), we have
\begin{equation*}
U(x,t;U_\mu^+)\leq  \frac{1}{1-\chi}\quad \forall\ x\in\R,\ t\geq 0.
\end{equation*}

Similarly, note that $U_\mu^+(x)\le\varphi_\mu(x)$.  Then by  comparison principle for parabolic equations and
Theorem \ref{super-sub-solu-thm}(1)  again, we have
\begin{equation*}
U(x,t;U_\mu^+)\leq \varphi_{\mu}(x) \ \quad \forall\ x\in\R\ t\geq 0.
\end{equation*}
Thus $U(\cdot,t;U_\mu^+)\leq U^{+}_{\mu}$. This complete of (i).

(ii)  For $0\leq t_{1}\leq t_{2}$, since
$$
U(\cdot,t_{2};U_\mu^+)=U(\cdot,t_{1},U(\cdot,t_{2}-t_{1};U_\mu^+))
$$
and by (i), $U(\cdot,t_{2}-t_{1};U_\mu^+)\leq U^{+}_{\mu} $, (ii) follows from comparison principle for parabolic equations.
\end{proof}

Let us define $U(x; u)$ to be
\begin{equation}
\label{U-eq}
{
U(x; u)=\lim_{t\to\infty}U(x,t; U^{+}_{\mu})=\inf_{t>0}U(x,t; U^{+}_{\mu}).
}
\end{equation}
 By the a priori estimates for parabolic equations, the limit in \eqref{U-eq} is uniform in $x$ in compact subsets of $\R$
and $U(\cdot;u)\in C_{\rm unif}^b(\R)$.
We shall provide sufficient hypothesis on the choice of $d$ to guarantee that the function $U(\cdot;u)$ constructed above is not identically  zero for each $u\in \mathcal{E}_{\mu}$.
Now, we are ready to prove that the function $u\in \mathcal{E}_{\mu}\to U(\cdot;u)\in\mathcal{E}_{\mu}$ for $d$ large enough.

\begin{lem}\label{lm2}
For every $0 <\chi<  \frac{1}{2}$, $0<\mu<\tilde{\mu}<\min\{1,2\mu, \mu+ \frac{1}{\mu+\sqrt{1-\mu^2}} \}$, there is { $d_{0}>1$ such that}
\begin{equation}
U(x;u)\geq\begin{cases}  U^{-}_{\mu}(x),\quad x\ge \underline{a}_{\mu,\tilde{\mu},d}\cr
U_\mu^-(x_\delta),\quad x\le x_\delta=\underline{a}_{\mu,\tilde{\mu},d}+\delta
\end{cases}
\end{equation}\label{Eq2 of Th2}
for every $u\in\mathcal{E}_{\mu}$,  $t\geq 0$, and $0<\delta\ll 1$,  whenever $d\geq d_{0}$.
\end{lem}

\begin{proof} Let $u\in \mathcal{E}_{\mu}$ be fixed.  Let $O=(\underline{a}_{\mu,\tilde{\mu},d},\infty)$.
Note that $U_\mu^-(\underline{a}_{\mu,\tilde{\mu},d})=0$. By Theorem \ref{super-sub-solu-thm}(2),
$U_\mu^-(x)$ is a sub-solution of \eqref{ODE2} on $O\times (0,\infty)$ for $d\ge d_0$.
Note also that $U_\mu^+(x)\ge U_\mu^-(x)$ for $x\ge \underline{a}_{\mu,\tilde{\mu},d}$ and $U(\underline{a}_{\mu,\tilde{\mu},d},t;U_\mu^+))>0$
for all $t\ge 0$. Then by comparison principle for parabolic equations, we have that
$$
U(x,t;U_\mu^+)\ge U_\mu^-(x)\quad \forall \,\, x\ge \underline{a}_{\mu,\tilde{\mu},d},\,\, t\ge 0
$$
for $d\ge d_0$.

Now for any $0<\delta\ll 1$, by Theorem \ref{super-sub-solu-thm}(2), $U(x,t)=U_\mu^-(x_\delta)$ is a sub-solution of
\eqref{ODE2} on $\R\times R$. Note that $U_\mu^+(x)\ge U_\mu^-(x_\delta)$ for $x\le x_\delta$ and
$U(x_\delta,t;U_\mu^+)\ge U_\mu^-(x_\delta)$ for $t\ge 0$. Then by comparison principle for parabolic equations again,
$$
U(x,t;U_\mu^+)\ge U_\mu^-(x_\delta)\quad \forall\,\, x\le x_\delta,\, \, t>0.
$$
The lemma then follows.
\end{proof}

\begin{rk}\label{Remark-lower-bound-for -solution}
 It follows from Lemmas \ref{lm1} and \ref{lm2} that if the assumptions of these two lemmas hold, then
$$
U_{\mu,\delta}^{-}(\cdot)\leq U(\cdot,t;U_\mu^+)\leq U^{+}_{\mu}(\cdot)$$
for every $u\in\mathcal{E}_{\mu}$, $t\geq0$ and $0\le \delta\ll 1$, where
$$
U_{\mu,\delta}^-(x)=\begin{cases}  U^{-}_{\mu}(x),\quad x\ge \underline{a}_{\mu,\tilde{\mu},d}+\delta\cr
U_\mu^-(x_\delta),\quad x\le x_\delta=\underline{a}_{\mu,\tilde{\mu},d}+\delta.
\end{cases}
$$
 This implies that $$
U_{\mu,\delta}^{-}(\cdot)\leq U(\cdot;u)\leq U^{+}_{\mu}(\cdot)$$
for every $u\in\mathcal{E}_{\mu}$. Hence  $u\in\mathcal{E}_{\mu}\mapsto U(\cdot;u)\in \mathcal{E}_{\mu}.$
\end{rk}

From now on, we suppose that $0<\mu,\chi<1$ are  fixed and satisfy inequality \eqref{Eq01_Th1}. Next choose $\tilde \mu$ such that
$$\mu<\tilde{\mu}<\min\{1, 2\mu, \mu+ \frac{1}{\mu+\sqrt{1-\mu^2}}\},$$
and take $d\geq d_{0}$,  where $d_{0}$ is given by Lemma \ref{lm2}.   We have the following important result.

\medskip

\begin{lem}\label{MainLem02}
Assume that $0<\mu,\chi<1$ satisfy
\eqref{Eq01_Th1}. Then for every $u\in \mathcal{E}_{\mu}$ the associated function $U(\cdot;u)$ satisfied the elliptic equation,
\begin{equation}\label{Eq_MainLem02}
0=U_{xx}+(c_{\mu}-\chi V'(x;u))U_{x}+(1-\chi V(x;u)-(1-\chi)U)U\quad \forall\,\, x\in\R.
\end{equation}
\end{lem}
\begin{proof}
Let $\{t_{n}\}_{n\geq 1}$ be an increasing sequence of positive real numbers converging to $\infty$. For every $n\geq 1$, define $U_{n}(x,t)=U(x,t+t_{n};u)$ for every $x\in\R, \ t\geq 0$.
For every $n$, $U_{n}$ solves the PDE
\begin{equation*}
\begin{cases}
\partial_{t}U_{n}=\partial_{xx}U_{n}+(c_{\mu}-\chi V'(x;u))\partial_{x}U_{n}+(1-\chi V(x;u)-(1-\chi)U_{n})U_{n}\ \ x\in\R, \ t>0\\
U_{n}(\cdot,0)=U(\cdot,t_{n};u).
\end{cases}
\end{equation*}

Let $\{T(t)\}_{t\geq 0}$ be the analytic semigroup on $C^{b}_{\rm unif}(\R)$ generated by $\Delta-I$
 and  let $X^{\beta}={\rm Dom}((I-\Delta)^{\beta})$ be the fractional power spaces of $I-\Delta$ on $C_{\rm unif}^b(\R)$ ($\beta\in [0,1]$).

 The variation of constant formula and the fact that $V''(x;u)-V(x;u)=-u(x)$ yield that
\begin{eqnarray}\label{variation -of-const}
U(\cdot,t;u)&=& T(t)U_{\mu}^{+}+ \int_{0}^{t}T(t-s)((c_{\mu}-\chi V'(\cdot;u))U_{x})ds\nonumber \\
& & +\int_{0}^{t}T(t-s)(2-\chi V(\cdot;u))U(\cdot,s;u)ds-(1-\chi)\int_{0}^{t}T(t-s)U^{2}(\cdot,s;u)ds\nonumber\\
&=& T(t)U_{\mu}^{+}+ \int_{0}^{t}T(t-s)(((c_{\mu}-\chi V'(\cdot;u))U)_{x}+\chi V''(\cdot;u)U(\cdot,s;u) )ds\nonumber \\
& & +\int_{0}^{t}T(t-s)(2-\chi V(\cdot;u))U(\cdot,s;u)ds-(1-\chi)\int_{0}^{t}T(t-s)U^{2}(\cdot,s;u)ds\nonumber\\
&=& \underbrace{T(t)U_{\mu}^{+}}_{I_{1}(t)}+ \underbrace{\int_{0}^{t}T(t-s)(((c_{\mu}-\chi V'(\cdot;u))U)_{x})ds}_{I_{2}(t)}\nonumber \\
& & +\underbrace{\int_{0}^{t}T(t-s)(2-\chi u)U(\cdot,s;u)ds}_{I_{3}(t)}-(1-\chi)\underbrace{\int_{0}^{t}T(t-s)U^{2}(\cdot,s;u)ds}_{I_{4}(t)}.
\end{eqnarray}
Let $0<\beta<\frac{1}{2}$ be fixed.  We have that
\begin{equation*}
\|I_{1}(t)\|_{X^{\beta}}\leq C_{\beta}t^{-\beta}e^{-t}\|U_{\mu}^{+}\|_{\infty}=\frac{C}{1-\chi}t^{-\beta}e^{-t}.
\end{equation*}
Next, using inequality  (3.1) in \cite{SaSh} , we have that
\begin{eqnarray*}
\|I_{2}(t)\|_{X^{\beta}} & \leq & C_{\beta}\int_{0}^{t}(t-s)^{-\frac{1}{2}-\beta}e^{-(t-s)}\| (c_{\mu}-\chi V'(\cdot;u))U(\cdot,s;u)\|_{\infty}\nonumber\\
& \leq & \frac{C_{\beta}}{1-\chi}(c_{\mu}+\frac{\chi}{1-\chi})\int_{0}^{t}(t-s)^{-\beta-\frac{1}{2}}e^{-(t-s)}ds\nonumber\\
&\leq & \frac{C_{\beta}}{1-\chi}(c_{\mu}+\frac{\chi}{1-\chi})\Gamma(\frac{1}{2}-\beta).
\end{eqnarray*}
And
\begin{eqnarray*}
\|I_{3}(t)\|_{X^{\beta}}& \leq & C_{\beta}\int_{0}^{t}(t-s)^{-\beta}e^{-(t-s)}\|(2-\chi u)U(\cdot,s;u)\|_{\infty}ds\nonumber\\
 & \leq & \frac{C_{\beta}}{1-\chi}(2+\frac{\chi}{1-\chi})\int_{0}^{t}(t-s)^{-\beta}e^{-(t-s)}ds\nonumber\\
 & \leq & \frac{C_{\beta}}{1-\chi}(2+\frac{\chi}{1-\chi})\Gamma(1-\beta).
\end{eqnarray*}
Similar arguments yield that
\begin{equation*}
\|I_{4}(t)\|_{X^{\beta}}\leq \frac{C_{\beta}}{(1-\chi)^{2}}\Gamma(1-\beta).
\end{equation*}
Therefore, for every $T>0$ we have that
\begin{equation}\label{Eq_Convergence01}
\sup_{t\geq T}\|U(\cdot,t;u)\|_{X^{\beta}}\leq M_{T}<\infty,
\end{equation}
where
\begin{equation}\label{Eq_Conv02}
M_{T}= \frac{C_{\beta}}{1-\chi}\Big[T^{-\beta}e^{-T}+ (c_{\mu}+\frac{1}{1-\chi})(2\Gamma(1-\beta)+\Gamma(\frac{1}{2}-\beta))\Big].
\end{equation}
Hence it follows that
\begin{equation}\label{Eqq000}
\sup_{n\geq 1, t\geq 0}\|U_{n}(\cdot,t)\|_{X^{\beta}}\leq M_{t_{1}}<\infty.
\end{equation}

Next, for every $t,h\geq 0$ and $n\geq 1$, we have that
\begin{equation}\label{Eqq00}
\|I_{1}(t+h+t_{n})-I_{1}(t+t_{n})\|_{X^{\beta}}\leq C_{\beta}h^{\beta}(t+t_{n})^{-\beta}e^{-(t+t_n)}\|U_{\mu}^{+}\|_{\infty}\leq C_{\beta}h^{\beta}t_{1}^{-\beta}e^{-t_1}\|U_{\mu}^{+}\|_{\infty},
\end{equation}
\begin{eqnarray}\label{Eqq01}
&&\|I_{2}(t+h+t_n)-I_{2}(t+t_n)\|_{X^{\beta}}\nonumber\\
& &\leq  \int_{0}^{t+t_n}\|(T(h)-I)T(t+t_n-s)( ( (c_{\mu}-\chi V'(\cdot,s;u))U(\cdot,s;u) )_{x} )\| _{X^{\beta}}ds\nonumber\\
& &\,\,\,  +\int_{t+t_n}^{t+t_n+h}\|T(t+t_n+h-s)( ( (c_{\mu}-\chi V'(\cdot,s;u))U(\cdot,s;u) )_{x} )\|_{X^{\beta}}ds\nonumber\\
& &\leq  C_{\beta}h^{\beta}\int_{0}^{t+t_n}(t+t_n-s)^{-\beta-\frac{1}{2}} e^{-(t+t_n-s)}\|(c_{\mu}-\chi V'(\cdot,s;u))U(\cdot,s;u) \|_{\infty}ds\nonumber\\
& & \,\,  + C_{\beta}\int_{t+t_n}^{t+t_n+h}(t+t_n +h-s)^{-\beta-\frac{1}{2}}e^{-(t+t_n +h-s)}\|(c_{\mu}-\chi V'(\cdot,s;u))U(\cdot,s;u)\|_{\infty}ds \nonumber\\
&& \leq \frac{C_{\beta}}{1-\chi}(c_{\mu}+\frac{\chi}{1-\chi})\Big[h^{\beta}\Gamma(\frac{1}{2}-\beta)+\int_{t+t_n}^{t+t_n+h}(t+t_n+h-s)^{-\beta-\frac{1}{2}}e^{-(t+t_n+h-s)}ds \Big]\nonumber\\
& &\leq \frac{C_{\beta}}{1-\chi}(c_{\mu}+\frac{\chi}{1-\chi})\Big[h^{\beta}\Gamma(\frac{1}{2}-\beta)+\frac{h^{\frac{1}{2}-\beta}}{\frac{1}{2}-\beta}\Big],
\end{eqnarray}
\begin{eqnarray}\label{Eqq02}
\|I_{3}(t+t_n+h)-I_{3}(t+t_n)\|_{X^{\beta}}& \leq & \int_{0}^{t+t_n}\|(T(h)-I)T(t+t_n-s)((2-\chi u)U(\cdot,s;u))\|_{X^{\beta}}ds\nonumber\\
& & +\int_{t+t_n}^{t+t_n+h}\|T(t+t_n+h-s)((2-\chi u)U(\cdot,s;u))\|_{X^{\beta}}ds \nonumber\\
&\leq &\frac{C_{\beta}}{1-\chi}(2+\frac{\chi}{1-\chi})\Big[h^{\beta}\Gamma(1-\beta)+\frac{h^{1-\beta}}{1-\beta} \Big] ,
\end{eqnarray}
and
\begin{eqnarray}\label{Eqq03}
\|I_{4}(t+t_n+h)-I_{4}(t+t_n)\|_{X^{\beta}}&\leq &\int_{0}^{t+t_n}\|(T(h)-I)T(t+t_n-s)U^{2}(\cdot,s;u)\|_{X^{\beta}}ds \nonumber\\
& &+\int_{t+t_n}^{t+t_n+h}\|T(t+t_n+h-s)U^{2}(\cdot,s;u)\|_{X^{\beta}}ds\nonumber\\
& \leq &\frac{C_{\beta}}{(1-\chi)^2}\Big[h^{\beta}\Gamma(1-\beta)+\frac{h^{1-\beta}}{1-\beta} \Big].
\end{eqnarray}
It follows from inequalities \eqref{Eqq000}, \eqref{Eqq00}, \eqref{Eqq01}, \eqref{Eqq02} and \eqref{Eqq03}, the functions $U_{n} : [0, \infty)\to X^{\beta}$ are uniformly bounded and equicontinuous.
Since $X^{\beta}$ is continuously imbedded in $C^{\nu}(\R)$ for every $0\leq \nu<2\beta$ (See \cite{Dan Henry}),
therefore, the Arzela-Ascoli Theorem and Theorem 3.15 in   \cite{Friedman}, imply that there is a function $\tilde{U}(\cdot,\cdot;u)\in C^{2,1}(\R\times(0,\infty))$ and a subsequence $\{U_{n'}\}_{n\geq 1}$ of $\{U_{n}\}_{n\geq 1}$ such that $U_{n'}\to \tilde{U}$ in $C^{2,1}_{loc}(\R\times(0, \infty))$ as $n\to \infty$ and $\tilde{U}(\cdot,\cdot;u)$ solves the PDE
$$
\begin{cases}
\partial_{t}\tilde{U}=\partial_{xx}\tilde{U}+(c_{\mu}-\chi V'(x;u))\partial_{x}\tilde{U}+(1-\chi V(x;u)-(1-\chi)\tilde{U})\tilde{U}\ \ x\in \R\ , \ t>0\\
\tilde{U}(x,0)=\lim_{n\to \infty}U(x,t_{n'};u).
\end{cases}
$$
But $U(x;u)=\lim_{t\to \infty}U(x,t;u)$ and $t_{n'}\to \infty$ as $n\to \infty$, hence $\tilde{U}(x,t;u)=U(x;u)$ for every $x\in \R,\ t\geq 0$. Hence $U(\cdot;u)$ solves \eqref{Eq_MainLem02}.
\end{proof}

\begin{lem}
\label{aux-lm} Assume that $0<\mu<1$ and $0<\chi<\frac{1}{2}$ satisfying that \eqref{Eq01_Th1}. Then, for any given $u\in\mathcal{E}_\mu$,
\eqref{Eq_MainLem02} has a unique bounded non-negative solution satisfying that
\begin{equation}
\label{aux-eq1}
\liminf_{x\to -\infty}U(x)>0\quad {\rm and}\quad \lim_{x\to\infty}\frac{U(x)}{e^{-\mu x}}=1.
\end{equation}
\end{lem}

\begin{proof}
First, note that for any two $U_1,U_1\in C_{\rm unif}^b(\R)$ satisfying \eqref{aux-eq1} and that $U_i(x)>0$ for $x\in\R$, we can define the so called
part metric $\rho(U_1,U_2)$ as follows:
$$
\rho(U_1,U_2)=\inf\{\L\ln \alpha\,|\, \alpha\ge 1,\,\, \frac{1}{\alpha} U_1(x)\le U_2(x)\le \alpha U_1(x)\quad \forall\,\, x\in\R\}.
$$
Moreover, there is $\alpha\ge 1$ such that
$$
\rho(U_1,U_2)=\ln \alpha\quad {\rm and}\quad \frac{1}{\alpha} U_1(x)\le U_2(x)\le \alpha U_1(x)\quad \forall\,\, x\in\R.
$$

Next, fix $u\in \mathcal{E}_\mu$. Suppose that $U_1(x)$ and $U_2(x)$ are two solutions of
\eqref{Eq_MainLem02} satisfying \eqref{aux-eq1}.
Let $\alpha\ge 1$ be such that $\rho(U_1,U_2)=\ln \alpha$. Note that $U(x,t;U_i)=U_i$ for all $t\ge 0$ and every $i=1,2$. Hence
$$
\rho(U(\cdot,t;U_1),U(\cdot,t;U_2))=\ln\alpha\quad \forall \,\, t\ge 0.
$$
Assume that $\alpha>1$.
Note that
\begin{equation*}
\frac{1}{\alpha} U_1(x)\le U_2(x)\le \alpha U_1(x)\quad \forall\,\, x\in\R
\end{equation*}
and
$$
(\alpha U_{i})_{t}> (\alpha U_{i})_{xx}+(c_{\mu}-\chi V'(\cdot;u))(\alpha U_{i})_{x}+(1-\chi V(\cdot;u)-(1-\chi)(\alpha U_{i}))(\alpha U_{i})
$$
for $i=1,2$. Thus comparison principle for parabolic equations implies that
\begin{equation}\label{aux-eq2}
\begin{cases}
U_{2}(x)\leq U(x,t,\alpha U_{1})< \alpha U_{1}(x)\quad \forall\ x\in\R,\ t> 0\cr
U_{1}(x)\leq U(x,t,\alpha U_{2})< \alpha U_{2}(x)\quad \forall\ x\in\R,\ t> 0.
\end{cases}
\end{equation}
Since $U_{i}(x)>0$ for every $x\in\R$ and $\lim_{x\to \infty}\frac{U_{i}(x)}{e^{-\mu x}}=1$ for each $i=1,2$, then  for every $1<\alpha'<\alpha$, there is $R_{\alpha'}\gg 1$ such that
\begin{equation}\label{aux-eq3}
U_{2}(x)< \alpha' U_{1}(x),\quad U_1(x)<\alpha' U_2(x) \quad \forall\ x\geq R_{\alpha'}.
\end{equation}
Since $U_{i}(x)>0$ for every $x\in\R$ and $\liminf_{x\to-\infty}U_{i}(x)>0$ for each $i=1,2$, then
\begin{equation}\label{aux-eq4}
l_{\alpha'}:=\min\{\inf_{x\leq R_{\alpha'}}U_{1}(x), \inf_{x\leq R_{\alpha'}}U_{2}(x)\}>0 , \ \ \quad \forall\ 1<\alpha'<\alpha.
\end{equation}
For every $1<\alpha'<\alpha$, $i=1,2$ and $x\leq R_{\alpha'}$, we have
\begin{eqnarray}\label{aux-eq5}
(\alpha U_{i})_{t}& = & (\alpha U_{i})_{xx}+(c_{\mu}-\chi V'(x;u))(\alpha U_{i})_{x}\nonumber\\
&& \,\, +(1-\chi V(x;u)-(1-\chi)(\alpha U_{i}))(\alpha U_{i}) +(1-\chi)(\alpha-1)U_{i}(\alpha U_{i})\nonumber\\
& \geq & (\alpha U_{i})_{xx}+(c_{\mu}-\chi V'(x;u))(\alpha U_{i})_{x}\nonumber\\
&& \,\, +(1-\chi V(x;u)-(1-\chi)(\alpha U_{i}))(\alpha U_{i}) +(1-\chi)(\alpha-1)l_{\alpha'}(\alpha U_{i}).
\end{eqnarray}
On the other hand, if we set $W^{i}(x,t)=e^{\varepsilon t}U(x,t;\alpha U_{i})$, it follows from \eqref{aux-eq2} that
\begin{eqnarray*}
W^i_{t}& = &\varepsilon W^i{ +e^{\varepsilon t}U_t(x,t;\alpha U_i)}\nonumber\\
     &= & \varepsilon W^i +W^i_{xx}+(c_\mu-\chi V'(x;u))W_x^i+(1-\chi V(x;u)-(1-\chi)W^i)W^i \nonumber\\
     && \,\, +(1-\chi)(e^{\varepsilon t}-1)U(x,t;\alpha U_{i})W^i\nonumber\\
&\leq & W^i_{xx}+(c_\mu -\chi V'(x;u))W_x^i+(1-\chi V(x;u)-(1-\chi)W^i)W^i +\varepsilon W^i \nonumber\\
&&\,\, +\alpha(1-\chi)(e^{\varepsilon t}-1)U_{i}W^{i}\nonumber\\
&\leq & W^i_{xx}+(c_\mu-\chi V'(x;u))W_x^i+(1-\chi V(x;u)-(1-\chi)W^i)W ^i\nonumber\\
&&\,\, +\Big(\varepsilon  +\alpha(1-\chi)(e^{\varepsilon t}-1)L_{\alpha'}\Big)W^{i},
\end{eqnarray*}
where
$$
L_{\alpha'}=\max\{\sup_{x\leq R_{\alpha'}}U_{1}(x),\sup_{x\leq R_{\alpha'}}U_{2}(x) \}.
$$
Choose $0<\varepsilon\ll 1$ such that
$$
\varepsilon  +\alpha(1-\chi)(e^{\varepsilon t}-1)L_{\alpha'}< (1-\chi)(\alpha-1)l_{\alpha'}\quad 0\leq t\leq 1.
$$
Thus, for $x\leq R_{\alpha'}$ and $0\leq t\leq 1$ we have
\begin{equation}\label{aux-eq6}
W^i_{t}\leq W^i_{xx}+(c_\mu-\chi V'(x;u))W_x^i+(1-\chi V(x;u)-(1-\chi)W^i)W^i+(1-\chi)(\alpha-1)l_{\alpha'}W^i.
\end{equation}
But inequality \eqref{aux-eq2} implies that
$U(R_{\alpha'},t;\alpha U_{i})<\alpha U_{i}(R_{\alpha'})$
for every $t>0$ and $i=1,2$. So, choose $0<\varepsilon\ll 1$ such that
\begin{eqnarray}\label{aux-eq7}
W^{i}(R_{\alpha'},t)=e^{\varepsilon t}U(R_{\alpha'},t;\alpha U_{i})\leq \alpha U_{i}(R_{\alpha'})\quad \frac{1}{2}\leq t\leq 1, \ \ i=1,2.
\end{eqnarray}
Therefore, using comparison principle for parabolic equations, it follows from inequalities \eqref{aux-eq5}, \eqref{aux-eq6} and \eqref{aux-eq7} that
$$
W^{i}(x,t)=e^{\varepsilon t}U(x,t;\alpha U_{i})\leq\alpha U_{i}(x)\quad \forall \ x\leq R_{\alpha'},\ \frac{1}{2}\leq t\leq 1, \ i=1,2.
$$
for $0<\varepsilon\ll 1$. Hence there is $0<\varepsilon_{0}\ll 1$ such that
$$
U(x,1;\alpha U_{i})\leq e^{-\varepsilon} \alpha U_{i}(x)\quad \forall \ x\leq R_{\alpha'}, \ i=1,2.
$$
Combining this with \eqref{aux-eq2}, we obtain
that
\begin{equation}\label{aux-eq8}
\begin{cases}
U_{2}(x)\leq e^{-\varepsilon_{0}}\alpha U_{1}(x)\quad x\leq R_{\alpha'}\cr
U_{1}(x)\leq e^{-\varepsilon_{0}}\alpha U_{2}(x)\quad x\leq R_{\alpha'}.
\end{cases}
\end{equation}
Combining inequalities \eqref{aux-eq4} and \eqref{aux-eq8} we have that
$$
\frac{1}{\max\{\alpha', e^{-\varepsilon_{0}}\alpha\}}U_{1}(x)\leq U_{2}(x)\leq \max\{\alpha', e^{-\varepsilon_{0}}\alpha\}U_{1}(x)\quad \forall x\in\R.
$$
From what it follows that
$$
\alpha\leq \max\{\alpha', e^{-\varepsilon_{0}}\alpha\}<\alpha,
$$
which is a contradiction. Hence $\alpha=1$ and then $U_1=U_2$. The lemma is thus proved.
\end{proof}

We now prove Theorem \ref{existence-tv-thm}.

\begin{proof}[Proof of Theorem \ref{existence-tv-thm}]
First of all,
 let us consider the normed linear space  $\mathcal{E}=C^{b}_{\rm unif}(\R)$ endowed with the norm
$$\|u\|_{\ast}=\sum_{n=1}^{\infty}\frac{1}{2^n}\|u\|_{L^{\infty}([-n,\ n])}. $$
For every $u\in\mathcal{E}_{\mu}$ we have that
$$\|u\|_{\ast}\leq \frac{1}{1-\chi}. $$
Hence $\mathcal{E}_{\mu}$ is a bounded convex subset of $\mathcal{E}$. Furthermore, since the convergence in $\mathcal{E}$ implies the pointwise convergence, then $\mathcal{E}_{\mu}$ is a closed, bounded, and convex subset of $\mathcal{E}$. Furthermore, a sequence of functions in $\mathcal{E}_{\mu}$ converges with respect to norm $\|\cdot\|_{\ast}$ if and only if it  converges locally uniformly convergence on $\R$.

We prove that the mapping $\mathcal{E}_\mu\ni u\mapsto U(\cdot;u)$ has a fixed point. We divide the proof in two steps.

\smallskip

\noindent {\bf Step 1.} In this step, we prove that the mapping $\mathcal{E}_\mu\ni u\mapsto U(\cdot;u)$ is compact.

 Let $\{u_{n}\}_{n\geq 1}$ be a sequence of elements of $\mathcal{E}_{\mu}$. Since $U(\cdot;u_{n})\in \mathcal{E}_{\mu}$ for every $n\geq 1$ then $\{U(\cdot;u_{n})\}_{n\geq 1}$ is clearly uniformly bounded by $\frac{1}{1-\chi}$. Using inequality \eqref{Eq_Convergence01}, we have that
\begin{equation*}
\sup_{t\geq 1}\|U(\cdot,t;u_{n})\|_{X^{\beta}}\leq M_{1}
\end{equation*}
for all $n\geq 1$ where $M_{1}$ is given by \eqref{Eq_Conv02}. Therefore there is $0<\nu\ll 1$ such that
\begin{equation}\label{Proof-MainTh3- Eq1}
\sup_{t\geq 1}\|U(\cdot,t;u_{n})\|_{C^{\nu}_{\rm unif}(\R)}\leq \tilde{M_{1}}
\end{equation} for every $n\geq 1$ where $\tilde{M_{1}}$ is a constant depending only on $M_{1}$. Since for every $n\geq 1$ and every $x\in\R$, we have that $U(x,t;u_{n})\to U(x;u_{n})$ as $t\to \infty,$ then it follows from \eqref{Proof-MainTh3- Eq1} that
\begin{equation}\label{Prof-MainTh3- Eq2}
\|U(\cdot;u_{n})\|_{C^{\nu}_{\rm unif}}\leq \tilde{M_{1}}
\end{equation} for every $n\geq 1$. Which implies that the sequence $\{U(\cdot;u_{n})\}_{n\geq 1}$ is equicontinuous. The Arzela-Ascoli's Theorem implies that there is a subsequence $\{U(\cdot;u_{n'})\}_{n\geq 1}$ of the sequence $\{U(\cdot;u_{n})\}_{n\geq 1}$ and a function $U\in C(\R)$ such that $\{U(\cdot;u_{n'})\}_{n\geq 1}$ converges to $U$ locally uniformly on $\R$. Furthermore, the function $U$ satisfies inequality \eqref{Prof-MainTh3- Eq2}. Combining this with the fact  $U_{\mu}^{-}(x)\leq U(x;u_{n'})\leq U_{\mu}^{+}(x)$ for every $x\in\R$ and $n\geq 1$, by letting $n$ goes to infinity, we obtain that  $U\in \mathcal{E}_{\mu}$.

\smallskip

\noindent{\bf Step 2.} In this step, we prove that the mapping $\mathcal{E}_\mu\ni u\mapsto U(\cdot;u)$ is continuous.

Let $u\in \mathcal{E}_{\mu}$ and $\{u_{n}\}_{n\geq 1}\in \mathcal{E}_{\mu}^{\N}$ such that $\|u_{n}-u\|_{\ast}\to 0$ as $n\to \infty$. Suppose by contradiction that $\|U(\cdot;u_{n})-U(\cdot;u)\|_{\ast}$ does not converge to zero. Hence there is $\delta>0$ and a subsequence $\{u_{n_{1}}\}_{n\geq 1}$ such that
\begin{equation}\label{Prof-MainTh3- Eq3}
\|U(\cdot;u_{n_{1}})-U(\cdot;u)\|_{\ast}\geq \delta \quad \forall \ n\geq 1.
\end{equation}
 For every $n\geq 1$, we have that $U(\cdot,u_{n_{1}})$ satisfies
 \begin{align}\label{Prof-MainTh3- Eq4}
 0=&U_{xx}(x;u_{n_{1}})+(c_{\mu}-\chi V(x;u_{n_{1}}))U_{x}(x;u_{n_{1}})\nonumber\\
 &\,\, +(1-\chi V(x;u_{n_{1}})-(1-\chi)U(x;u_{n_{1}}))U(x;u_{n_{1}})\quad \forall\,\, x\in\R.
 \end{align}

{\bf Claim 1.} $\|V(\cdot;u_{n})-V(\cdot;u)\|_{\ast}\to 0$ as $n\to \infty$. Indeed, for every $R>0$, it follows from \eqref{Inverse of u} that
\begin{align}\label{Prof-MainTh3- Eq5}
|V(x;u_{n})-V(x;u)| & \leq  \frac{1}{\sqrt{\pi}}\int_{0}^{\infty}\int_{\R}e^{-s}e^{-z^2}|u_{n}(x-2\sqrt{t}z)-u(x-2\sqrt{s}z)|dzds\nonumber\\
& \leq \frac{1}{\sqrt{\pi}}\int_{0}^{R}\int_{B(0,R)}e^{-s}e^{-z^2}|u_{n}(x-2\sqrt{t}z)-u(x-2\sqrt{s}z)|dzds   \nonumber\\
& + \frac{2}{(1-\chi)\sqrt{\pi}}\int_{ \{ s\geq R\  \text{or}\ |z|\geq R\}}e^{-s}e^{-z^2}dzds.
\end{align}
Thus for every $k\in\N$ and every $R>1$, we have that
\begin{align}\label{Prof-MainTh3- Eq6}
\|V(\cdot;u_{n})-V(\cdot;u)\|_{L^{\infty}([-k\ ,k ])}& \leq \frac{1}{\sqrt{\pi}}\Big[ \int_{0}^{R}\int_{B(0,R)}e^{-s}e^{-z^2}dzds \Big] \|u_{n}-u\|_{L^{\infty}([-(k+2R^{\frac{3}{2}})\ ,\ (k+2R^{\frac{3}{2}}) ])}\nonumber\\
&\,\,\,\, +  \frac{2}{(1-\chi)\sqrt{\pi}}\int_{ \{ s\geq R\  \text{or}\ |z|\geq R\}}e^{-s}e^{-z^2}dzds   \nonumber\\
& \leq  \frac{1}{\sqrt{\pi}}\underbrace{\Big[ \int_{0}^{\infty}\int_{\R}e^{-s}e^{-z^2}dzds \Big]}_{\sqrt{\pi}} \|u_{n}-u\|_{L^{\infty}([-(k+2R^{\frac{3}{2}})\ ,\ (k+2R^{\frac{3}{2}}) ])}\nonumber\\
&\,\,\,\, +  \frac{2}{(1-\chi)\sqrt{\pi}}\int_{ \{ s\geq R\  \text{or}\ |z|\geq R\}}e^{-s}e^{-z^2}dzds   \nonumber\\
& \leq  2^{k+2R^2}\|u_{n}-u\|_{\ast} + \frac{2}{(1-\chi)\sqrt{\pi}}\int_{ \{ s\geq R\  \text{or}\ |z|\geq R\}}e^{-s}e^{-z^2}dzds.
\end{align}
Now, let $\varepsilon>0$ be given. Choose $R\gg 1$ and $k\gg 1$ such that
\begin{equation}\label{Prof-MainTh3- Eq7}
 \frac{2}{(1-\chi)\sqrt{\pi}}\int_{ \{ s\geq R\  \text{or}\ |z|\geq R\}}e^{-s}e^{-z^2}dzds  <\frac{\varepsilon}{3}  \qquad \text{and} \qquad \sum_{i\geq k}\frac{2}{(1-\chi)2^{i}} < \frac{\varepsilon}{3}.
\end{equation}
Next, choose $N\gg 1$ such that
\begin{equation}\label{Prof-MainTh3- Eq8}
2^{k+2R^2}\|u_{n}-u\|_{\ast}<\frac{\varepsilon}{3} \quad \forall\ n\geq N.
\end{equation}
It follows from inequalities  \eqref{Prof-MainTh3- Eq6}, \eqref{Prof-MainTh3- Eq7} and \eqref{Prof-MainTh3- Eq8} that for every $n\geq N $, we have
\begin{eqnarray}\label{Prof-MainTh3- Eq9}
\|V(\cdot;u_{n})-V(\cdot;u)\|_{\ast} & \leq & \sum_{i\geq k}\frac{1}{2^{i}}\|V(\cdot;u_{n})-V(\cdot;u)\|_{L^{\infty}([-i\ ,\ i])}+\|V(\cdot;u_{n})-V(\cdot;u)\|_{L^{\infty}([-k\ ,\ k])}\nonumber\\
& \leq &\sum_{i\geq k}\frac{2}{(1-\chi)2^{i}} + \|V(\cdot;u_{n})-V(\cdot;u)\|_{L^{\infty}([-k\ ,\ k])} < \varepsilon.
\end{eqnarray}
Thus, the claim follows.

\smallskip

\noindent {\bf Claim 2.} $\|V'(\cdot;u_{n})-V'(\cdot;u)\|_{\ast}\to 0$ as $n\to \infty$. Indeed, it follows from \eqref{Inverse of u} that
\begin{align}\label{Prof-MainTh3- Eq10}
V'(x;w)&=\int_{0}^{\infty}\int_{\R}\frac{(z-x)e^{-s}}{2s\sqrt{4\pi s}}e^{-\frac{|x-z|^{2}}{4s}}w(z)dzds\nonumber\\
&= \frac{-1}{\sqrt{\pi}}\int_{0}^{\infty}\int_{\R}ye^{-s}e^{-y^2}w(x-2\sqrt{s}y)dzds\ \ \forall \ x\in\R \ , w\in C^{b}_{\rm unif}(\R).
\end{align}
Since
$$
\lim_{R\to \infty}\int_{\{s\geq R\ \text{or}\ |y|\geq R\}}|y|e^{-s}e^{-y^2}dzds =0,
$$ same arguments as in the proof of Claim 1 yield Claim 2.

Now, since $V''(\cdot;u_{n})-V''(\cdot;u)=(V(\cdot;u_{n})-V(\cdot;u))-(u_{n}-u)$, it follows from Claim 1 that
\begin{equation}\label{Prof-MainTh3- Eq11}
\|V''(\cdot;u_{n})-V''(\cdot;u)\|_{\ast}\to 0 \quad \text{as}\quad  n\to \infty.
\end{equation}
Combining inequality \eqref{Prof-MainTh3- Eq2}, Claim 1, Claim 2, \eqref{Prof-MainTh3- Eq11}, Theorem 3.15  of \cite{Friedman}, and the Arzela-Ascoli's Theorem, there is a subsequence $\{U(\cdot;u_{n_{2}})\}_{\geq 1}$ of  $\{U(\cdot;u_{n_{1}})\}_{n\geq 1}$ and a function $U\in C^{2}(\R)$ such that $\{U(\cdot;u_{n_{2}})\}_{\geq 1}$ converges to U in  $C^{2}_{loc}(\R^N)$  and $U$ satisfies
\begin{equation}\label{Prof-MainTh3- Eq12}
0=U_{xx}+(c_{\mu}-\chi V'(x;u))U_{x}+(1-\chi V(x; u)-(1-\chi)U)U.
\end{equation} Hence $U\in \mathcal{E}_{\mu}$ and
\begin{equation}\label{Prof-MainTh3- Eq13}
\|U(\cdot;u_{n_{2}})-U\|_{\ast}\to 0 \quad \text{as}\quad n\to \infty.
\end{equation}
But
\begin{equation}\label{Prof-MainTh3- Eq14}
0=U_{xx}(x;u)+(c_{\mu}-\chi V'(x;u))U_{x}(x;u)+(1-\chi V(x; u)-(1-\chi)U(x;u))U(x;u) \quad \forall\,\,x\in\R.
\end{equation}
By Lemma \ref{aux-lm}, $U(\cdot)=U(\cdot;u)$. By \eqref{Prof-MainTh3- Eq3},
$$
\|U(\cdot)-U(\cdot;u)\|\ge \delta,
$$
which is a contradiction. Hence  the mapping $\mathcal{E}_\mu\ni u\mapsto U(\cdot;u)$ is continuous.

Now by Schauder's Fixed Point Theorem, there is $U\in\mathcal{E}_\mu$ such that $U(\cdot;U)=U(\cdot)$. Then
$(U(x),V(x;U))$ is a stationary solution of \eqref{stationary-eq} with $c=c_\mu$. It is clear that
$$
\lim_{x\to\infty}\frac{U(x)}{e^{-\mu x}}=1.
$$

We claim that if $\chi<\frac{1}{2}$, then
$$
\lim_{x\to -\infty}U(x)=1.
$$
For otherwise, we may assume that there is $x_n\to -\infty$ such that $U(x_n)\to a\not =1$ as $n\to\infty$. Define $U_{n}(x)=U(x+x_{n})$ for every $x\in\R$ and $n\geq 1$. By observing that $U_{n}=U(\cdot;U_{n})$ for every n, hence it follows from the step 1, that there is a subsequence $\{U_{n'}\}_{n\geq 1}$ of $\{U_{n}\}_{n\geq}$ and a function $U^*\in\mathcal{E}_{\mu}$ such that $\|U_{n'}-U^*\|_{\ast}\to 0$ as $n\to \infty$. Next, it follows from step 2 that $(U^*,V(\cdot;U^*))$ is also a stationary solution of \eqref{stationary-eq}.

\smallskip

\noindent {\bf Claim 3.} $\inf_{x\in\R}U^{*}(x)>0$. Indeed, let $0< \delta\ll 1$ be fixed. For  every $x\in\R$, there $N_{x}\gg 1$ such that $x+x_{n'}< x_{\delta}$ for all $n\geq N_{x}$. Hence, It follows from Remark \ref{Remark-lower-bound-for -solution} that  $$0< U_{\mu}^{-}(x_{\delta})\leq U(x+x_{n'}) \ \forall\ n\geq N_{n}.$$
Letting $n$ goes to infinity in the last inequality, we obtain that $U_{\mu}^{-}(x_{\delta})\leq U^{*}(x)$ for every $x\in\R$. The claim thus follows.

Since $\chi<\frac{1}{2}$, it follows from Theorem 1.8 of \cite{SaSh} that $U^*(x)=V(x;U^*)=1 $ for every $x\in \R$. In particular, $a=U^*(0)=1$, which is a contradiction.
This implies that $U^*(0)=1=a$, which is a contradiction. Hence $\lim_{x\to -\infty} U(x)=1$.
\end{proof}
As a direct consequence of Theorem \ref{existence-tv-thm} we present the proof of Theorem A.

\medskip

\begin{proof}[Proof of Theorem A]  Let $0<\chi<\frac{1}{2}$ be fixed.
According to Theorem \ref{existence-tv-thm}, it is enough to show  that for every $c\geq c^*(\chi)$ there is $0<\mu(c)<1$ with $c_{\mu(c)}=c$ and $\mu(c)$ satisfies \eqref{Eq01_Th1}.
 To this end, recall that there is a unique $\mu^*(\chi)\in (0, 1)$ such that
 $$
 \frac{\mu^*(\chi)\Big(\mu^* (\chi)+\sqrt{1-(\mu^{*}(\chi))^2}\Big)}{1-(\mu^{*}(\chi))^2}=\frac{1-\chi}{\chi}.
 $$
 Recall also that $c^*(\chi):=c_{\mu^*(\chi)}=\mu^*(\chi)+\frac{1}{\mu^*(\chi)}$. Since the function $(0, 1)\ni\mu\mapsto c_{\mu}=\mu+\frac{1}{\mu}$ is continuous, decreasing with $\lim_{\mu\to 0^+}c_{\mu}=\infty$, then for every $c\geq c^{*}(\chi)$, there is a unique $\mu(c)\in(0, \mu^*(\chi)]$ such that $c=c_{\mu(c)}$. Furthermore, we have that
 $$\frac{\mu(c)\Big(\mu(c)+\sqrt{1-\mu(c)^2}\Big)}{1-\mu(c)^{2}}\leq \frac{\mu^*(\chi)\Big(\mu^*(\chi)+\sqrt{1-(\mu^{*}(\chi))^2}\Big)}{1-(\mu^{2}(\chi))^2}=\frac{1-\chi}{\chi}.
  $$
  Hence, applying Theorem \ref{existence-tv-thm} the result follows.
\end{proof}

\section{Spatial spreading speeds}

In this section, we study the spreading properties of solutions of \eqref{IntroEq0-2} with nonnegative initial functions $u_0$ which have nonempty and compact supports, and prove Theorem B. Throughout this section, we assume that $0<\chi<1$, unless specified otherwise.

  One important ingredient in the proof of
Theorem B is to prove that for
any $u_{0}\in C^{+}_{c}(\R)$, there is $M>0$
such that
$$
0\le  u(x,t;u_0)\le  M e^{-\mu^* (|x|-c_{\mu^*} t)},
$$
where $(u(x,t;u_0),v(x,t;u_0))$ is the solution of \eqref{IntroEq0-2}  with $u(x,0;u_0)=u_0(x)$, $\mu^*$ is as in \eqref{mu-star-new-eq}, and   $c_{\mu^*}=\mu^*+\frac{1}{\mu^*}$.
 To this end, we first prove some lemmas.

Fix $u_{0}\in C_{c}^{+}(\R)$ and $0< \chi <1$.
{  Let $R\gg 1$ such that ${\rm supp}(u_{0})\subset[-R, R]$.  Recall that
for every $0<\mu<1$, $\varphi_{\mu}(x)=e^{-\mu x}$.
For every $T>0$ and $\mu\in(0, 1)$,  we define
$$
\mathcal{E}^T_{\mu}(u_{0}):=\{ u\in C_{\rm unif}^{b}(\R\times[0,T])\, | \, 0\leq u\leq \bar{U}_{\mu}\ \text{and}\ u(\cdot,0)=u_0\},
$$
$$
\bar{U}_{\mu}(x,t):=Me^{\mu c_{\mu}t}\varphi_{\mu}(|x|)=Me^{\mu c_{\mu}t}\min\{\varphi_{\mu}(-x),\varphi_{\mu}(x)\}, \, \forall \ x\in\R,\ t\ge 0,
$$
and $$M:= \max\{\frac{e^{R}}{1-\chi},e^{R}\|u_{0}\|_{\infty}\}. $$
Observe that $M$ is independent of $T$,  $\bar{U}_{\mu}(x,0)\geq u_{0}(x)$ for every $x\in \R$, and
\begin{equation}\label{Spreading-speed-eq02}
\partial_{t}\bar{U}_\mu(x,t)-\partial_{xx}\bar{U}_\mu(x,t)-\bar{U}_\mu(x,t)=0, \quad \forall\ x\neq 0, \ t\ge 0.
\end{equation}

For given $u\in \mathcal{E}_{\mu}^{T}(u_0)$, let $V(x,t;u)$ be the solution of the second equation in \eqref{IntroEq0-2}. Note that
\begin{equation}\label{Inverse of u-new}
V(x,t;u)=\int_{0}^{\infty}\int_{\R}\frac{e^{-s}}{\sqrt{4\pi s}}e^{-\frac{|x-z|^{2}}{4s}}u(z,t)dzds,
\end{equation}
In what follows, some of the arguments are similar to those of the previous sections. Hence, some details might be omitted. The next Lemma is an equivalent of Lemmas \ref{Mainlem2} and \ref{Mainlem3}, whence it provides  pointwise estimates on $V(\cdot,t;u)$ and $|\partial_{x} V(\cdot,t;u)|$ for every $u\in\mathcal{E}_{\mu}^{T}(u_0)$.

\begin{lem}\label{Spread-speed-v-v'-est}
 For every { $0<\mu<1$ and} for every $u\in \mathcal{E}_{\mu}^{T}(u_{0})$ we have that
\begin{equation}\label{general-spraeding-eq2-new}
V(\cdot,\cdot;u)\leq  { \frac{1}{1-\mu^2}}\bar{U}_\mu(\cdot,\cdot)
\end{equation}
and
\begin{equation}\label{general-spraeding-eq3-new}
|\partial_{x_{i}}V(\cdot,\cdot;u)|\leq  { \frac{\mu+\sqrt{1-\mu^2}}{1-\mu^{2}}}\bar{U}_{\mu}(\cdot,\cdot).
\end{equation}
\end{lem}

\begin{proof} Using the fact that  $\bar{U}_{\mu}(z,t)\leq e^{-\mu (z-c_{\mu}t)}$ for every $z\in\R, t\geq 0$, we obtain that
\begin{eqnarray}
\label{general-spraeding-eq5-new}
\int_{0}^{\infty}\int_{\R}\frac{e^{-s}}{(4\pi s)^{\frac{1}{2}}}e^{-\frac{|x-z|^2}{4s}}\bar{U}_{\mu}(z,t)dzds &\le &\int_{0}^{\infty}\int_{\R}\frac{e^{-s}}{(4\pi s)^{\frac{1}{2}}}e^{-\frac{|x-z|^2}{4s}} Me^{\mu c_{\mu}t}e^{-\mu z}dzds
\nonumber\\
&= &\frac{Me^{\mu c_{\mu}t}}{\pi^{\frac{1}{2}}}\int_{0}^{\infty}\int_{\R}e^{-s}e^{-|y|^2}e^{-\mu(x+2\sqrt{s}y)}dyds\nonumber\\
&=& \frac{Me^{-\mu x+\mu c_{\mu}t}}{\pi^{\frac{1}{2}}}\int_{0}^{\infty}e^{-s}\Big[\int_{\R}e^{-(y+\mu \sqrt{s})^2+\mu^{2}s}dy\Big]ds\nonumber\\
&= & Me^{-\mu x+\mu c_{\mu}t}\int_{0}^{\infty}e^{-(1-\mu^{2})s}ds\nonumber\\
&=& \frac{Me^{-\mu x+\mu c_{\mu}t}}{1-\mu^2}\quad \forall\,\, x\in\R.
\end{eqnarray}
Similarly,  Using the fact that  $\bar{U}_{\mu}(z,t)\leq e^{\mu(z+c_{\mu}t)}$ for every $z\in\R, t\geq 0$, we obtain that
\begin{eqnarray}
\label{general-spraeding-eq5-new-1}
\int_{0}^{\infty}\int_{\R}\frac{e^{-s}}{(4\pi s)^{\frac{1}{2}}}e^{-\frac{|x-z|^2}{4s}}\bar{U}_{\mu}(z,t)dzds &\le &\int_{0}^{\infty}\int_{\R}\frac{e^{-s}}{(4\pi s)^{\frac{1}{2}}}e^{-\frac{|x-z|^2}{4s}} Me^{\mu c_{\mu}t}e^{\mu z}dzds
\nonumber\\
&= &\frac{Me^{\mu c_{\mu}t}}{\pi^{\frac{1}{2}}}\int_{0}^{\infty}\int_{\R}e^{-s}e^{-|y|^2}e^{\mu(x+2\sqrt{s}y)}dyds\nonumber\\
&=& \frac{Me^{\mu x+\mu c_{\mu}t}}{\pi^{\frac{1}{2}}}\int_{0}^{\infty}e^{-s}\Big[\int_{\R}e^{-(y-\mu \sqrt{s})^2+\mu^{2}s}dy\Big]ds\nonumber\\
&=& \frac{Me^{\mu x+\mu c_{\mu}t}}{1-\mu^2}\quad \forall\,\, x\in\R.
\end{eqnarray}
Thus,
\begin{eqnarray}\label{general-spraeding-eq5-new}
\int_{0}^{\infty}\int_{\R}\frac{e^{-s}}{(4\pi s)^{\frac{1}{2}}}e^{-\frac{|x-z|^2}{4s}}\bar{U}_{\mu}(z,t)dzds&\le & \frac{1}{1-\mu^2}\bar{U}_{\mu}(x,t).
\end{eqnarray}
Since $V(\cdot,\cdot;u)\leq V(\cdot,\cdot;\bar{U}_{\mu})$, hence inequality \eqref{general-spraeding-eq2-new} follows from \eqref{general-spraeding-eq5-new}.

\smallskip

For every $x\in\R$ and $t>0$, we have that
\begin{align}\label{general-spraeding-eq6-new}
|\partial_{x}V(x,t;u)|&=\Big|\int_{0}^{\infty}\int_{\R}\frac{(y-x)}{2s(4\pi s)^{\frac{1}{2}}}e^{-s}e^{-\frac{|x-y|^2}{4s}}u(y,t)dyds\Big|\nonumber\\
&\le \int_{0}^{\infty}\int_{\R}\frac{|y-x|}{2s(4\pi s)^{\frac{1}{2}}}e^{-s}e^{-\frac{|x-y|^2}{4s}}u(y,t)dyds\nonumber\\
&\le \int_{0}^{\infty}\int_{\R}\frac{|y-x|}{2s(4\pi s)^{\frac{1}{2}}}e^{-s}e^{-\frac{|x-y|^2}{4s}}\bar U_\mu (y,t)dyds\nonumber\\
&\le \int_{0}^{\infty}\int_{\R}\frac{|y-x|}{2s(4\pi s)^{\frac{1}{2}}}e^{-s}e^{-\frac{|x-y|^2}{4s}}Me^{\mu c_{\mu}t}e^{-\mu y}dyds\nonumber\\
&=  \frac{Me^{-\mu x+\mu c_{\mu}t}}{\pi^{\frac{1}{2}}}\int_{0}^{\infty}\int_{\R}\frac{|y|}{\sqrt{s}}e^{-(1-\mu^2)s}e^{-(y+\mu\sqrt{s})^{2}}dyds  \nonumber\\
& \leq  \Big(\frac{\sqrt{\pi}}{\sqrt{1-\mu^2}}+\frac{\mu\sqrt{\pi}}{1-\mu^2}\Big)Me^{-\mu x+\mu c_{\mu}t}\quad {\rm by}\quad \eqref{Eq_Mainlem0002}).
\end{align}
Similarly, every $x\in\R$ and $t>0$, w have that
\begin{align}\label{general-spraeding-eq7-new}
|\partial_{x}V(x,t)|&=\Big|\int_{0}^{\infty}\int_{\R}\frac{(y-x)}{2s(4\pi s)^{\frac{1}{2}}}e^{-s}e^{-\frac{|x-y|^2}{4s}}u(y,t)dyds\Big|\nonumber\\
&\le \int_{0}^{\infty}\int_{\R}\frac{|y-x|}{2s(4\pi s)^{\frac{1}{2}}}e^{-s}e^{-\frac{|x-y|^2}{4s}}\bar U_\mu (y,t)dyds\nonumber\\
&\le \int_{0}^{\infty}\int_{\R}\frac{|y-x|}{2s(4\pi s)^{\frac{1}{2}}}e^{-s}e^{-\frac{|x-y|^2}{4s}}Me^{\mu c_{\mu}t}e^{\mu y}dyds\nonumber\\
&=  \frac{Me^{\mu x+\mu c_{\mu}t}}{\pi^{\frac{1}{2}}}\int_{0}^{\infty}\int_{\R}\frac{|y|}{\sqrt{s}}e^{-(1-\mu^2)s}e^{-(y-\mu\sqrt{s})^{2}}dyds  \nonumber\\
& \leq  \Big(\frac{\sqrt{\pi}}{\sqrt{1-\mu^2}}+\frac{\mu\sqrt{\pi}}{1-\mu^2}\Big)Me^{\mu x+\mu c_{\mu}t}\quad {\rm by}\quad \eqref{Eq_Mainlem0002}).
\end{align}
Combining \eqref{general-spraeding-eq6-new} and \eqref{general-spraeding-eq7-new}, we obtain \eqref{general-spraeding-eq3-new}. The lemma is thus proved.
\end{proof}

Now, for every $u\in\mathcal{E}^{T}_{\mu}(u_0),$\  let $\bar{U}(\cdot,\cdot;u)$ be the solution of the initial value problem
\begin{equation}\label{Spreading-speed-eq04}
\begin{cases}
\partial_{t} \bar{U}=\bar{\mathcal{L}}(\bar{U}), \qquad  x\in\R,\ t>0\\
\bar{U}(\cdot,0;u)=u_{0}
\end{cases}
\end{equation}
where $$\bar{\mathcal{L}}(\bar{U}):=\partial_{xx}\bar{U}-\chi \partial_{x}V(\cdot,\cdot;u)\partial_{x}\bar{U}(\cdot,\cdot)+(1-\chi V(\cdot,\cdot;u)+(1-\chi)\bar{U})\bar{U}. $$
For given $\tilde U_0\in\R$, let $\tilde U(t;\tilde U_0)$ be the solution of the initial value problem
\begin{equation}
\label{new-ode-eq}
\tilde U_t=\tilde U(1-(1-\chi)\tilde U)
\end{equation}
with $\tilde U(t;\tilde U_0)=\tilde U_0$.

The next Lemma shows that $
\bar{U}(\cdot,\cdot;u)\in\mathcal{E}_{\mu}^T(u_{0})$ for every
$u\in\mathcal{E}_{\mu}^T(u_{0})$  for $\mu=\mu^*$.

\begin{lem}\label{Spread-speed-sup-sol}
 Assume that $0<\chi<1$. Let $\mu^{*}\in (0, 1)$ satisfy \eqref{mu-star-new-eq}.
Then for every  $u\in\mathcal{E}_{\mu^{*}}^{T}(u_0)$ we have that
\begin{equation}\label{general-spraeding-eq8-new}
0\leq \bar{U}(\cdot,\cdot;u)\leq \bar{U}_{\mu^{*}}(\cdot,\cdot)
\end{equation}
\end{lem}

\begin{proof}
Since $u_{0}\geq 0$, comparison principle for parabolic equations
implies that $\bar{U}(\cdot,\cdot;u)\geq 0$. Observe that
$$
\partial_{t}\bar{U}(\cdot,\cdot;u)\leq  \partial_{xx}\bar{U}(\cdot,\cdot;u)-\chi\partial_{x} V(\cdot,\cdot;u)\partial_{x} \bar{U}(\cdot,\cdot;u) +(1-(1-\chi)\bar{U}(\cdot,\cdot;u))\bar{U}(\cdot,\cdot;u).
$$
Thus, comparison principle implies that  $\bar{U}(\cdot,t;u)\leq \tilde {U}(t;\|u_0\|_{\infty})$. Hence
\begin{equation}\label{general-spraeding-eq9-new}
\bar{U}(0,t;u)\leq \tilde U(t;\|u_0\|_{\infty})\leq \max\{\frac{1}{1-\chi},\|u_0\|_{\infty}\}\leq \bar{U}_{\mu^*}(0,t) \ \forall\ t\geq 0.
\end{equation}
If we restrict $\bar{U}_{\mu^*}$ on $\R^{N}\setminus\{0\}$, using \eqref{Spreading-speed-eq02} and \eqref{general-spraeding-eq3-new},  we obtain that
\begin{eqnarray}\label{general-spraeding-eq10-new}
\partial_{t}\bar{U}_{\mu^{*}}-\bar{\mathcal{L}}(\bar{U}_{\mu^{*}})& = &\mu\chi \partial_{x} V(\cdot,\cdot;u)\bar{U}_{\mu^{*}}+(\chi V(\cdot,\cdot;u)+(1-\chi)\bar{U}_{\mu^{*}})\bar{U}_{\mu^{*}}\nonumber\\
&\geq & -\mu\chi |\partial_{x} V(\cdot,\cdot;u)|\bar{U}_{\mu^{*}}  +(\chi V(\cdot,\cdot;u)+(1-\chi)\bar{U}_{\mu^{*}})\bar{U}_{\mu^{*}}\nonumber\\
&\geq & -\frac{\mu^{*}\chi(\mu^{*} +\sqrt{1-\mu^{*2}}}{1-\mu^{*2}})\Big[\bar{U}_{\mu^{*}}\Big]^{2}  +(1-\chi)\Big[\bar{U}_{\mu^{*}}\Big]^{2} \nonumber\\
& =& \chi \underbrace{\Big[\frac{\mu^{*}(\mu^{*} +\sqrt{1-\mu^{*2}}) }{1-\mu^{*2}}-\frac{1-\chi}{\chi} \Big]}_{=0}\Big[\bar{U}_{\mu^{*}}\Big]^{2}.
\end{eqnarray}
Combining inequalities \eqref{general-spraeding-eq9-new}, \eqref{general-spraeding-eq10-new} with the fact that $u_0\leq \bar U_{\mu^*}(\cdot,0)$, thus comparison principle for parabolic equations implies that $\bar{U}(\cdot,\cdot;u)\leq \bar U_{\mu^{*}}(\cdot,\cdot)$, which complete the proof of the lemma.
\end{proof}


\begin{lem}
\label{spraeding-speed-important-lm} Assume that $0<\chi<1$. For any
given $u_0\in C_c^+(\R)$ and  $T>0$,  $u(\cdot,\cdot;u_0)\in
\mathcal{E}^T_{\mu}(u_{0})$ with $\mu=\mu^{*}$.
\end{lem}

\begin{proof} In this proof, we put $\mu=\mu^*$.
Consider the normed linear space  $\mathcal{E}^{T}=C^{b}_{\rm unif}(\R\times[0,T])$ endowed with the norm
$$\|u\|_{\ast,T}=\sum_{n=1}^{\infty}\frac{1}{2^n}\|u\|_{L^{\infty}([-n,\ n]\times[0, T])}. $$
For every $u\in\mathcal{E}_{\mu}^{T}(u_{0})$ we have that $\|u\|_{\ast,T}\leq M e^{\mu c_{u}T}. $
Hence $\mathcal{E}_{\mu}^{T}(u_{0})$ is a bounded convex subset of $\mathcal{E}^{T}$. Since the convergence in $\mathcal{E}^{T}$ implies the pointwise convergence, then $\mathcal{E}_{\mu}^{T}(u_{0})$ is a closed, bounded, and convex subset of $\mathcal{E}^T$. Furthermore, a sequence of functions in $\mathcal{E}^T_{\mu}(u_{0})$ converges with respect to norm $\|\cdot\|_{\ast,T}$ if and only if it  converges locally uniformly on $\R\times[0, T]$.

By Lemma \ref{Spread-speed-sup-sol}, for any $u\in  \mathcal{E}_{\mu}^T(u_{0})$, $\bar{U}(\cdot,\cdot;u)\in\mathcal{E}_{\mu}^T(u_{0})$.
  We prove the lemma by showing that $u(\cdot,\cdot;u_0)$ is a fixed point of  the mapping
$ \mathcal{E}_{\mu}^T(u_{0})\ni u\mapsto \bar{U}(\cdot,\cdot;u)\in\mathcal{E}_{\mu}^T(u_{0})$ and divide the proof into two steps.

\smallskip

\noindent {\bf Step 1.} In this step, we prove that the mapping $\mathcal{E}^T_\mu(u_0)\ni u\mapsto \bar{U}(\cdot,\cdot;u)\in \mathcal{E}^T_\mu(u_0)$ is compact.

Indeed, let $\{u_{n}\}_{n\geq 1}\subset \mathcal{E}_{\mu}^{T}(u_0)$ be given. For every $n\geq 1$, $\bar{U}(\cdot,\cdot;u_n)$ satisfied
$$
\begin{cases}
\partial_{t}\bar{U}(\cdot,\cdot,u_{n})=\partial_{xx}\bar{U}(\cdot,\cdot,u_{n})-\chi \partial_{x}V(\cdot,\cdot;u_{n})\partial_{x}\bar{U}(\cdot,\cdot;u_{n})+(1-\chi V(\cdot,\cdot;u_{n})-(1-\chi)\bar{U}(\cdot,\cdot;u_n))\bar{U}(\cdot,\cdot;u_{n})\\
\bar{U}(\cdot,0;u_n)=u_{0}
\end{cases}
$$

Taking $\{T(t)\}_{t\geq 0}$ to be the analytic semigroup generated by $A:=(\Delta-I)$ on $C^{b}_{\rm unif}(\R)$, the variation of constant formula and similar arguments to the one used to establish \eqref{variation -of-const} yield that for every $t\geq 0$,
\begin{eqnarray}\label{Spread-speed-eq07}
\bar{U}(\cdot,t;u_n)
& =& T(t)u_{0} -\chi\underbrace{\int_{0}^{t}T(t-s)\partial_{x}\Big( \partial_{x}V(\cdot,s;u_{n})\bar{U}(\cdot,s;u_{n})\Big)ds}_{I^{n}_{1}(t)}\nonumber\\
& &+\underbrace{\int_{0}^{t}T(t-s)((2-\chi u_{n}(\cdot,s)-(1-\chi)\bar{U}(\cdot,s;u_n))\bar{U}(\cdot,s;u_{n}))ds}_{I^{n}_{2}(t)}\nonumber\\
\end{eqnarray}

 \noindent{\bf Claim 1.} For every $0\le \beta<\frac{1}{2}$ and every $0<K<T$, the functions $[K,\ T]\ni t\mapsto \bar{U}(\cdot,t;u_{n})\in X^{\beta}$ are uniformly bounded and  equicontinuous.

Indeed, let $0\leq \beta< \frac{1}{2}$ and  $K\leq t\leq t+h\leq T $ be fixed.  Then
\begin{eqnarray}\label{Spread-speed-eq08}
\|\bar{U}(\cdot,t;u_{n})-T(t)u_{0}\|_{X^{\beta}}&\leq & \chi\int_{0}^{t}\|T(t-s)\partial_{x}(\partial_{x}V(\cdot,s;u_{n})\bar{U}(\cdot,s;u_{n})) \|_{X^{\beta}}ds\nonumber\\
& & + \int_{0}^{t}\| T(t-s)((2-\chi u_{n}-(1-\chi)\bar{U}(\cdot,s;u_{n}))\bar{U}(\cdot,s;u_{n}))\|_{X^{\beta}}ds\nonumber\\
&\leq & \chi C_{\beta}\int_{0}^{t}(t-s)^{-\beta-\frac{1}{2}}e^{-(t-s)}\|u_{n}\|_{\infty}\|\bar{U}(\cdot,s;u_{n})) \|_{\infty}ds\nonumber\\
& & + C_{\beta}\int_{0}^{t}(t-s)^{-\beta}e^{-(t-s)}\|(2-\chi u_{n}-(1-\chi)\bar{U}(\cdot,s;u_{n}))\bar{U}(\cdot,s;u_{n})\|_{\infty}ds\nonumber\\
&\leq & C_{\beta}M^{2}(3+2\chi)e^{2\mu c_{\mu}T}\Big[\int_{0}^{t}(t-s)^{-\beta-\frac{1}{2}}e^{-(t-s)}ds+\int_{0}^{t}(t-s)^{-\beta}e^{-(t-s)}ds\Big]\nonumber\\
&\leq & C_{\beta}M^{2}(3+2\chi)e^{2\mu c_{\mu}T}\Big(\frac{t^{\frac{1}{2}-\beta}}{\frac{1}{2}-\beta}+\frac{t^{1-\beta}}{1-\beta}\Big).
\end{eqnarray}
This combining with the fact that $\|T(t)u_{0}\|_{X^{\beta}}\leq C_{\beta}t^{-\beta}e^{-t}\|u_{0}\|_{\infty}$ yield  that
\begin{equation}
\sup_{K\leq t\leq T}\|\bar{U}(\cdot,t;u_{n})\|_{X^{\beta}}\leq \bar{M}_{K},
\end{equation}
where
\begin{equation*}
\bar{M}_{K}=C_{\beta}K^{-\beta}e^{-K}+C_{\beta}M^{2}(3+2\chi)e^{2\mu c_{\mu}T}\Big(\frac{T^{\frac{1}{2}-\beta}}{\frac{1}{2}-\beta}+\frac{T^{1-\beta}}{1-\beta}\Big).
\end{equation*}
On the other hand, we have that
\begin{equation}\label{Spread-speed-eq009}
\|T(t+h)u_0-T(t)u_0\|_{X^{\beta}}\leq C_{\beta} h^{\beta}t^{-\beta}e^{-t}\|u_0\|_{\infty}\leq C_{\beta} h^{\beta}K^{-\beta}e^{-K}\|u_0\|_{\infty},
\end{equation}
and
\begin{eqnarray}\label{Spread-speed-eq09}
\|I^{n}_{1}(t+h)-I^{n}_{1}(t)\|_{X^{\beta}}& \leq & \int_{0}^{t}\|(T(h)-I)T(t-s)\partial_{x}(\partial_{x}V(\cdot,s;u_{n})\bar{U}(\cdot,s;u_{n}))\|_{\beta}ds\nonumber\\
& & + \int_{t}^{t+h}\|T(t+h-s)\partial_{x}(\partial_{x}V(\cdot,s;u_{n})\bar{U}(\cdot,s;u_{n}))\|_{\beta}ds\nonumber\\
& \leq & C_{\beta}h^{\beta}\int_{0}^{t}(t-s)^{-\beta-\frac{1}{2}}e^{-(t-s)}\|\bar{U}_{\mu}(\cdot,s)\|_{\infty}^2ds\nonumber\\
& & + C_{\beta}\int_{t}^{t+h}(t+h-s)^{-\frac{1}{2}-\beta}e^{-(t+h-s)}\|\bar{U}_{\mu}(\cdot,s)\|_{\infty}^2 \nonumber\\\nonumber\\
& \leq & M^{2}e^{2\mu c_{\mu}T}(C_{\beta}\Gamma(\frac{1}{2}-\beta)+ \frac{C}{\frac{1}{2}-\beta})(h^{\beta}+h^{\frac{1}{2}-\beta}).\nonumber\\
\end{eqnarray}

Similarly, we have
\begin{eqnarray}\label{Spread-speed-eq10}
\|I^{n}_{2}(t+h)-I^{n}_{2}(t)\|_{X^{\beta}}&\leq & \int_{0}^{t}\|(T(h)-I)T(t-s)(2-\chi u_{n}(\cdot,s)-(1-\chi)\bar{U}(\cdot,s;u_{n}))\bar{U}(\cdot,s,u_{n})\|_{X^{\beta}}ds\nonumber\\
& & + \int_{t}^{t+h}\|T(t+h-s)(2-\chi u_{n}(\cdot,s)-(1-\chi)\bar{U}(\cdot,s;u_{n}))\bar{U}(\cdot,s;u_{n})\|_{X^{\beta}}ds\nonumber\\
&\leq& C_{\beta}(\Gamma(1-\beta)+\frac{1}{1-\beta})(3+2\chi)M^{2}e^{2\mu c_{\mu}T}(h^{\beta}+h^{1-\beta})\nonumber\\
\end{eqnarray}

Thus, it follows from inequalities \eqref{Spread-speed-eq009}, \eqref{Spread-speed-eq09} and \eqref{Spread-speed-eq10} that
\begin{equation}
\|\bar{U}(\cdot,t+h;u_{n})-\bar{U}(\cdot,t;u_{n})\|_{X^{\beta}}\leq \bar{C}_{\beta,\chi,M,T}(h^{\beta}+h^{\frac{1}{2}-\beta}+h^{1-\beta}).
\end{equation}
Which complete the proof of Claim 1. It follows from Claim 1, the fact that $X^{\beta}$ is continuously embedded in $C^{\nu}(\R)$ for $0\leq \nu<2\beta$, Arzela-Ascili's Theorem and Theorem 3.15 in \cite{Friedman}, there is a function $\bar{U}\in C^{2,1}(\R\times(0,T])$ and a subsequence $\{u_{n_1}\}_{n\geq 1}$ of $\{u_{n}\}_{n\geq 1}$ such that $\{\bar{U}(\cdot,\cdot;u_{n_{1}})\}_{n\geq 1}$ converging locally uniformly to $\bar{U}$ in $C^{2,1}(\R\times(0,T])$.

\smallskip

\noindent {\bf Claim 2 :} $\lim_{t\to 0^{+}}\bar{U}(\cdot,t)=u_{0}$ in $C_{\rm unif}^{b}(\R)$

Indeed, let $\varepsilon>0$ be fixed. There is $0<t_{\varepsilon}<T$ such that
$$
\|T(t)u_{0}-u_0\|_{\infty} +C_{\beta}M^{2}(3+2\chi)e^{2\mu c_{\mu}T}( 2\sqrt{t}+t)<\varepsilon, \ \ \forall \ t\leq t_{\varepsilon}.
$$
Thus, by taking $\beta=0$, it follows from inequality \eqref{Spread-speed-eq08} that
\begin{eqnarray*}
\|\bar{U}(\cdot,t;u_{n_{1}})-u_{0}\|_{\infty} \leq  \|T(t)u_{0}-u_0\|_{\infty} +C_{\beta}M^{2}(3+2\chi)e^{2\mu c_{\mu}T}( 2\sqrt{t}+t)< \varepsilon
\end{eqnarray*}
for every $0\ \leq t<t_{\varepsilon}$. Hence, letting $n$ goes to infinity in the last inequality, we obtain that
$$
\|\bar{U}(\cdot,t)-u_{0}\|_{\infty}\leq \varepsilon, \ \forall\ 0<t<t_{\varepsilon}.
$$
Thus Claim 2 is proved. It is clear that $\bar{U}\in\mathcal{E}_{\mu}^T(u_0)$. Thus complete the proof of step 1.

\smallskip

\noindent {\bf Step 2:} In this step, we prove that the mapping $\mathcal{E}_{\mu}^{T}(u_{0})\ni u\mapsto \bar{U}(\cdot,\cdot;u)\in \mathcal{E}_{\mu}^T(u_{0})$ is continuous.

Indeed, let $\{u_{n}\}_{n\geq 1}\in \mathcal{E}_{\mu}^{T}(u_{0}) $ and $u\in \mathcal{E}_{\mu}^{T}(u_{0})$ such that $\|u_{n}-u\|_{\ast,T}\to 0$ as $n\to \infty$. Same arguments used in the proof of Claims 1 and 2 of step 2 in the proof of Theorem \ref{existence-tv-thm}, yield that
$$
\|V(\cdot,\cdot,u_{n})-V(\cdot,\cdot,u)\|_{\infty}+\|\partial_{x}V(\cdot,\cdot,u_{n})-\partial_{x}V(\cdot,\cdot,u)\|_{\infty}\to 0 \quad \text{as}\quad n\to\infty.
$$
Suppose by contradiction that there is $\delta>0$ and a subsequence $\{u_{n_{1}}\}_{n\geq 1}$ of $\{u_{n}\}_{n\geq 1}$
such that
\begin{equation}\label{Spread-speed-eq11}
\|\bar{U}(\cdot,\cdot;u_{n_{1}})-\bar{U}(\cdot,\cdot;u)\|_{\ast,T}\geq \delta, \quad \forall n\geq 1.
\end{equation}
From the proof of step 1, we know that there is a subsequence $\{u_{n_{2}}\}_{n\geq 1}$ of $\{u_{n_{1}}\}_{n\geq 1}$ and a function $\bar{U}\in \mathcal{E}_{\mu}^{T}(u_0)\cap C^{2,1}(\R\times(0,T])$ such that $ \{\bar{U}(\cdot,\cdot,u_{n_{2}})\}_{n\geq 1}$ converges to $\bar{U}$ in $C^{2,1}_{loc}(\R\times(0, T])$ with $\lim_{t\to 0^+}\|\bar{U}(\cdot,t)-u_0\|_{\infty}=0$.  But for each $n\geq 1$, we have
$$
\partial_{t}\bar{U}(\cdot,\cdot,u_{n_{2}})=\partial_{xx}\bar{U}(\cdot,\cdot,u_{n_{2}})-\chi \partial_{x}V(\cdot,\cdot;u_{n_{2}})\partial_{x}\bar{U}(\cdot,\cdot;u_{n_{2}})+(1-\chi V(\cdot,\cdot;u_{n_{2}})-(1-\chi)\bar{U}(\cdot,\cdot;u_{n_{2}}))\bar{U}(\cdot,\cdot;u_{n_{2}})
$$
Letting n goes to infinity in this equation and using the fact that $\lim_{t\to 0^+}\|\bar{U}(\cdot,t)-u_0\|_{\infty}=0$ , we obtain
\begin{equation}\label{Spread-speed-eq12}
\begin{cases}
\partial_{t}\bar{U}(\cdot,\cdot)=\partial_{xx}\bar{U}(\cdot,\cdot)-\chi \partial_{x}V(\cdot,\cdot;u)\partial_{x}\bar{U}(\cdot,\cdot)+(1-\chi V(\cdot,\cdot;u)-(1-\chi)\bar{U}(\cdot,\cdot))\bar{U}(\cdot,\cdot;u), \ \ x\in \R, 0<t\leq T\\
\bar{U}(\cdot,0)=u_{0}
\end{cases}
\end{equation}
Since $\bar{U}(\cdot,\cdot,u)$ is the only classical solution of \eqref{Spread-speed-eq12}, then
$ \bar{U}(\cdot,\cdot)=\bar{U}(\cdot,\cdot;u)$. Hence $$\lim_{n\to \infty}\|\bar{U}(\cdot,\cdot;u_{n_{2}})-\bar{U}(\cdot,\cdot;u)\|_{\ast,T}=0,$$ which contradicts \eqref{Spread-speed-eq11} .

\smallskip

Now, by Steps 1, 2, and Schauder's fixed point Theorem, there is a function $\bar{u}\in \mathcal{E}_{\mu}^{T}(u_{0})$ such that $\bar{U}(\cdot,\cdot;\bar{u})=\bar{u}$. The function $\bar{u}$ is clearly a classical solution of the PDE
\begin{equation} \label{Spread-speed-eq13}
\begin{cases}
\bar{u}_{t}=\bar{u}_{xx}-\chi V_{x}(\cdot,\cdot;\bar{u})\bar{u}_{x}+(1-\chi V(\cdot,\cdot;\bar{u}) -(1-\chi)\bar{u})\bar{u}, \quad x\in\R, 0<t\leq T\\
\bar{u}(\cdot,0)=u_0.
\end{cases}
\end{equation}
But we know from Theorem 1.1 in \cite{SaSh} that $u(\cdot,\cdot;u_0)$ is the only classical solution of \eqref{Spread-speed-eq13}. Hence $u(\cdot,\cdot;u_0)=\bar{u}\in \mathcal{E}_{\mu}^{T}(u_0)$.
\end{proof}

We are  now ready to prove Theorem B.

\begin{proof}[Proof of Theorem B]

For given $u_0\in C_c^+(\R)$, let $c_{\rm low}^*(u_0)$ be the largest positive number such that \eqref{spreading-eq1} holds
and $c_{\rm up}^*(u_0)$ be the smallest positive number such that \eqref{spreading-eq2} holds.
From the definition of $c_-^*(\chi)$ and $c_+^*(\chi)$, we have that
\begin{equation}\label{Spread-speed-eq01}
c_-^*(\chi)=\inf\{c^*_{\rm low}(u_{0})\, | \, u_{0}\in C_{c}^{+}(\R)\}\quad \text{and}\quad c_+^*(\chi)=\sup\{c^*_{\rm up}(u_{0})\, |\, u_{0}\in C_{c}^{+}(\R)\}.
\end{equation}

(i) Assume $0<\chi<1$. We first prove that $c_+^*(\chi)\le
c^*(\chi)$.

Let $\mu^*\in (0,1)$ satisfy \eqref{mu-star-new-eq}.  By Lemma 4.3, we have that
$$
0\leq u(x,t;u_0)\leq Me^{\mu^{*}(c_{\mu^{*}}t-|x|)}, \quad \forall\ x\in\R,\ t\geq 0,
$$
where $c_{\mu^*}=\mu^*+\frac{1}{\mu^*}$. Thus for every $c>c_{\mu^{*}}$ we have that
$$
0\leq \sup_{|x|\geq ct}u(x,t;u_0)\leq Me^{\mu^{*}(c_{\mu^{*}}-c)t}\to 0 \quad \text{as}\quad t\to \infty.
$$
Since $c_{\mu^{*}}$ is independent of $u_{0}$ thus
$c_{+}^{*}(\chi)\leq c_{\mu^{*}}=\mu^{*}+\frac{1}{\mu^{*}}= c^*(\chi)$.

Next, we prove that $c_+^*(\chi)\le 2+\frac{\chi}{1-\chi}$.  We know
that $u(\cdot,t;u_{0})\leq \tilde{U}(t;\|u_{0}\|_{\infty})\to
\frac{1}{1-\chi}$ as $t\to \infty$. Let $0<\varepsilon\ll 1$ be
fixed. Thus there is $T_{\varepsilon}>0$ such that $\tilde U(t;\|u_{0}\|_{\infty})\leq \frac{1}{1-\chi} +\varepsilon$ for
every $t\geq T_{\varepsilon}$.  We obtain that
\begin{equation}\label{Spread-speed-eq14}
\|u(\cdot,t;u_{0})\|_{\infty}, \,\,\, \|v(\cdot,t;u_0)\|_\infty,\,\,\,  \|v_x(\cdot,t;u_0)\|_\infty \leq  \frac{1}{1-\chi}+\varepsilon, \quad \forall t\geq T_{\varepsilon}.
\end{equation}
Let
$$M_{\varepsilon}=\sup_{0\le t\le T_{\varepsilon}}\|v_x(\cdot,t;u_0)\|_\infty.$$
Let $M>0$ be such that
$$
u_0(x)\le M e^{-|x|}.
$$
Consider
\begin{equation}
\label{new-aux-eq1} u_t=u_{xx}-\chi M_\epsilon u_x +u(1-{ (1-\chi)
u}),\quad x\in\R.
\end{equation}
Let $u^+(x,t)=Me^{-\big(x-(2+\chi M_\epsilon)t\big)}$. It is not difficult to see that $u^+(x,t)$ a super-solution of \eqref{new-aux-eq1}. This implies together $u_x^+(t,x)<0$ implies that
$u^+(x,t)$ is a super-solution of
\begin{equation}
\label{new-aux-eq0}
u_t=u_{xx}-\chi v_x(x,t;u_0)u_x+u(1-\chi v(x,t;u_0)-(1-\chi)u),\quad x\in\R
\end{equation}
on $t\in [0,T_{\varepsilon}]$ and then
$$
u(x,t;u_0)\le Me^{-\big(x-(2+\chi M_\epsilon)t\big)}\quad x\in\R,\,\, t\in [0,T_{\varepsilon}].
$$

Now consider
\begin{equation}
\label{new-aux-eq2}
u_t=u_{xx}-\Big(\frac{\chi}{1-\chi}+\varepsilon \chi\Big)u_x+u(1-(1-\chi)u),\quad x\in\R.
\end{equation}
It is not difficult to prove that
$\tilde u^+(x,t)=Me^{-\big(x-(2+\frac{\chi}{1-\chi}+\varepsilon \chi)(t-T_{\varepsilon})\big)} e^{(2+\chi M_\varepsilon)T_{\varepsilon}}$ is a super-solution
of \eqref{new-aux-eq2} for $t\ge T_{\varepsilon}$. This together with $\tilde u^+_x(x,t)<0$ and \eqref{Spread-speed-eq14} implies that $\tilde u^+(x,t)$
is a super-solution of \eqref{new-aux-eq0} on $[T_{\varepsilon},\infty)$. It then follows that
\begin{equation}
\label{new-aux-eq3}
u(x,t;u_0)\le M e^{-\big(x-(2+\frac{\chi}{1-\chi}+\varepsilon \chi)(t-T_{\varepsilon})\big)} e^{(2+\chi M_\varepsilon)T_{\varepsilon}},\quad x\in\R,\,\, t\ge T_{\varepsilon}.
\end{equation}

Similarly, we can prove that
\begin{equation}
\label{new-aux-eq4}
u(x,t;u_0)\le M e^{\big(x+(2+\frac{\chi}{1-\chi}+\varepsilon \chi)(t-T_{\varepsilon})\big)} e^{(2+\chi M_\varepsilon)T_{\varepsilon}},\quad x\in\R,\,\, t\ge T_{\varepsilon}.
\end{equation}
By \eqref{new-aux-eq3} and \eqref{new-aux-eq4}, we have
$$
c_+^*(\chi)\le 2+\frac{\chi}{1-\chi}+\varepsilon \chi.
$$
Letting $\varepsilon \to 0$, we get $c_+^*(\chi)\le
2+\frac{\chi}{1-\chi}$. Then (i) follows.

\smallskip

\smallskip

(ii)  Assume that $0<\chi<\frac{2}{3+\sqrt 2}$. We claim that
$$0< \underline{c}_{\chi}:=2\sqrt{1-\frac{\chi}{1-\chi}}-\frac{\chi}{1-\chi}\leq c_{\rm low}^*(u_{0})$$
 for every $u_{0}\in C^{+}_{c}(\R)$.

Indeed, let $u_{0}\in C_{c}^{+}(\R)$ be fixed. It follows from \cite[Lemmas 5.3 and  5.4]{SaSh} and the proof of \cite[Theorem 1.9 (i)]{SaSh} that
\begin{equation}\label{Spread-speed-eq15}
c_{\rm low}^*(u_{0})\geq \lim_{R\to\infty}\inf_{|x|\ge R, t\ge R}\Big[2\sqrt{1-\chi v(x,t;u_{0})}-\chi |\partial_{x}v(x,t;u_{0}))| \Big]>0.
\end{equation}
Since $0<\chi<\frac{2}{3+\sqrt{2}}$, then we have that $\underline c_{\chi}>0$. Inequality \eqref{Spread-speed-eq14} combined  with the fact that $\max\{\|v(\cdot,t;u_{0})\|_{\infty},\|v_{x}(\cdot,t;u_{0})\|_{\infty}\}\leq \|u(\cdot,t;u_{0})\|$ for every $t>0$, yield that
\begin{equation*}
2\sqrt{1-\chi v(x,t;u_{0})} -\chi|v_{x}(x,t;u_{0})|\geq 2\sqrt{1-\chi(\frac{1}{1-\chi}+\varepsilon)}-\chi(\frac{1}{1-\chi}+\varepsilon),\quad  t\geq T_{\varepsilon}, \ |x|\geq  T_{\varepsilon}.
\end{equation*}
Therefore
\begin{equation}\label{Spread-speed-eq16}
\lim_{R\to \infty}\inf_{|x|\geq R,t\geq R}2\sqrt{1-\chi v(x,t;u_{0})} -\chi|v_{x}(x,t;u_{0})|\geq 2\sqrt{1-\chi(\frac{1}{1-\chi}+\varepsilon)}-\chi(\frac{1}{1-\chi}+\varepsilon),\quad \forall\ 0<\varepsilon\ll 1.
\end{equation} Letting $\varepsilon$ goes to $0$ in \eqref{Spread-speed-eq16} and using \eqref{Spread-speed-eq15}, we obtain that
$$c^*_{\rm low}(u_{0})\geq \underline c_{\chi} .$$
We  then have $c_-^*(\chi)\ge \underline c_{\chi}$ and (ii) follows.
\end{proof}

\section{Spreading speeds and traveling waves on $\R^N$}

In this section, we consider the spatial spreading speeds and traveling wave solutions of \eqref{IntroEq0-2-new} with $N\ge 1$
and prove Theorems C and D. The proofs are based the ideas used in the proofs of Theorems A and B and some results in Theorems A and B.
We will skip the details of those arguments which are similar to some arguments in Theorems A and B.

\medskip

Throughout this section, we assume that $0<\chi<1$.
We call an entire solution $(u(x,t),v(x,t))$ of \eqref{IntroEq0-2-new}
a {\it traveling wave solution} of \eqref{IntroEq0-2-new}which connects $(1,1)$ and $(0,0)$ and propagates in the direction of $\xi\in S^{N-1}$ with speed
$c$ if there is $(\Phi(\cdot),\Psi(\cdot))\in C_{\rm unif}^b(\R)\times C_{\rm unif}^b(\R)$ such that $(u(x,t),v(x,t))=(\Phi(x\cdot\xi-ct),\Psi(x\cdot \xi-ct))$ and $(\Phi(-\infty),\Psi(-\infty))=(1,1)$, $(\Phi(\infty),\Psi(\infty))=(0,0)$.

\medskip

For given $x=(x_1,x_2,\cdots,x_N)\in\R^N$, let $|x|=\sqrt{x_1^2+x_2^2+\cdots+x_N^2}$.
Let
$$
C_c^+(\R^N)=\{u\in C_{\rm unif}^b(\R^N)\,|\, u(x)\ge 0,\,\, {\rm supp}(u)\,\,\, \text{is non-empty and compact}\}.
$$
Let
$$
C_{-}^*(\chi)=\{c^*_->0\,|\,
 \lim_{t\to\infty} \sup_{|x|\le ct} \big[|u(x,t;u_0)-1|+|v(x,t;u_0)-1|\big]=0\quad \forall\,\, u_0\in C_c^+(\R^N),\,\, \forall\,  0<c<c_{-}^*\}
 $$
 and
 $$
 C_+^*(\chi)=\{c^*_+>0\,|\,
 \lim_{t\to\infty}\sup _{|x|\ge ct} \big[ u(x,t;u_0)+v(x,t;u_0)\big]=0\quad \forall\,\, u_0\in C_c^+(\R^N),\,\, \forall\, c>c_{+}^*\}.
 $$
 Let
 $$
 c_{-}^*(\chi)=\sup\{c\in C_-^*(\chi)\}\quad {\rm and}\quad c_+^*(\chi)=\inf\{c\in C_+^*(\chi)\},
 $$
 where $c_-^*(\chi)=0$ if $C_-^*(\chi)=\emptyset$ and $c_+^*(\chi)=\infty$ if $C_+^*(\chi)=\emptyset$.
We call $[c_-^*(\chi),c_+^*(\chi)]$ the {\it spreading speed
interval } of \eqref{IntroEq0-2-new}.
\medskip

\begin{proof}[Proof of Theorem C]  Assume that $0<\chi<\frac{1}{2}$ and  that $c^*(\chi)$ is as in Theorem A.
For given $c\ge c^*(\chi)$, let $(u,v)=(U(x-ct),V(x-ct))$ be the traveling wave solution of \eqref{IntroEq0-2} connecting $(1,1)$ and $(0,0)$ with speed $c$.
It is then easy to verify that
$$(u(x,t),v(x,t)):=(U(x\cdot\xi-ct),V(x\cdot\xi-ct))
$$
is a traveling wave solution of \eqref{IntroEq0-2-new} which connects $(1,1)$ and $(0,0)$ and propagates in the direction of $\xi\in S^{N-1}$ with speed
$c$. This proves Theorem C.
\end{proof}

\begin{proof}[Proof of Theorem D]
(i)
Fix $u_{0}\in C_{c}^{+}(\R^N)$ and $0< \chi <1$.
 Let $R\gg 1$ such that ${\rm supp}(u_{0})\subset[-R, R]^N
 $.
For every $T>0$ and $\mu\in(0, 1)$,  we define
$$
\mathcal{E}^T_{\mu}(u_{0}):=\{ u\in C_{\rm unif}^{b}(\R^N\times[0,T])\, | \, 0\leq u\leq \bar{U}_{\mu}\ \text{and}\ u(\cdot,0)=u_0\},
$$
$$
\bar{U}_{\mu}(x,t):=Me^{-\mu\big(|x|- c_{\mu}t\big)}, \, \forall \ x\in\R^N,\ t\ge 0,
$$
and $$M:= \max\{\frac{e^{R}}{1-\chi},e^{R}\|u_{0}\|_{\infty}\}. $$
Observe that $M$ is independent of $T$,  $\bar{U}_{\mu}(x,0)\geq u_{0}(x)$ for every $x\in \R$, and
\begin{equation}\label{general-spreading-eq1}
\partial_{t}\bar{U}_\mu(x,t)-\Delta \bar{U}_\mu(x,t)-\bar{U}_\mu(x,t)=\frac{\mu(N-1)\bar{U}_{\mu}(x,t)}{|x|}\ge 0, \quad \forall\ x\neq 0, \ t\ge 0.
\end{equation}

For given $u\in \mathcal{E}_{\mu}^{T}(u_0)$, let  $V(x,t;u)$ be the solution of the second equation in \eqref{IntroEq0-2-new}. Note that
 \begin{equation}\label{general-spraeding-eq4}
V(x,t;u)=\int_{0}^{\infty}\int_{\R^N}\frac{e^{-s}}{(4\pi s)^{\frac{N}{2}}}e^{-\frac{|x-z|^{2}}{4s}}u(z,t)dzds.
\end{equation}

\noindent{\bf Claim 1.} {\it  For every { $0<\mu<\frac{1}{\sqrt N}$ and} for every $u\in \mathcal{E}_{\mu}^{T}(u_{0})$ we have that
\begin{equation}\label{general-spraeding-eq2}
V(\cdot,\cdot;u)\leq { \frac{2^{N}}{1-N\mu^2}}\bar{U}_\mu(\cdot,\cdot)
\end{equation}
and
\begin{equation}\label{general-spraeding-eq3}
|\partial_{x_{i}}V(\cdot,\cdot;u)|\leq { \frac{2^{N}(\mu+\sqrt{1-N\mu^2})}{1-N\mu^{2}}}\bar{U}_{\mu}(\cdot,\cdot).
\end{equation}
}

The claim can be proved by the arguments similar to those in Lemma \ref{Spread-speed-v-v'-est}. We provide some indication of the proof in the following.

Using the fact that $|a+b|\geq |a|-|b| $ for every $a, b\in\R^N$, we obtain that
\begin{eqnarray*}
\int_{0}^{\infty}\int_{\R^N}\frac{e^{-s}}{(4\pi s)^{\frac{N}{2}}}e^{-\frac{|x-y|^2}{4s}}\bar{U}_{\mu}(y,t)dyds& = & \frac{1}{\pi^{\frac{N}{2}}}\int_{0}^{\infty}\int_{\R^N}e^{-s}e^{-|y|^2}\bar{U}_{\mu}(x+2\sqrt{s}y,t)dyds\nonumber\\
&= &\frac{e^{\mu c_{\mu}t}}{\pi^{\frac{N}{2}}}\int_{0}^{\infty}\int_{\R^N}e^{-s}e^{-|y|^2}e^{-\mu|x+2\sqrt{s}y|}dyds\nonumber\\
&\leq &\frac{\bar{U}_{\mu}(x,t)}{\pi^{\frac{N}{2}}}\int_{0}^{\infty}\int_{\R^N}e^{-s}e^{-|y|^2}e^{2\mu \sqrt{s}|y|}dyds.\nonumber\\
&\leq &\frac{\bar{U}_{\mu}(x,t)}{\pi^{\frac{N}{2}}}\int_{0}^{\infty}\int_{\R^N}e^{-s}e^{-|y|^2}e^{2\mu \sqrt{s}\sum_{i}^{N}|y_i|}dyds.\nonumber\\
&=& \frac{\bar{U}_{\mu}(x,t)}{\pi^{\frac{N}{2}}}\int_{0}^{\infty}e^{-s}\Pi_{i=1}^{N}\Big[\int_{\R}e^{-(|y_{i}|-\mu \sqrt{s})^2+\mu^{2}s}dy_{i}\Big]ds.
\end{eqnarray*}
Hence
\begin{eqnarray}
\label{general-spraeding-eq5}
\int_{0}^{\infty}\int_{\R^N}\frac{e^{-s}}{(4\pi s)^{\frac{N}{2}}}e^{-\frac{|x-z|^2}{4s}}\bar{U}_{\mu}(z,t)dzds& \le& \frac{\bar{U}_{\mu}(x,t)}{\pi^{\frac{N}{2}}}\int_{0}^{\infty}e^{-s}\Pi_{i=1}^{N}\Big[2e^{\mu^2 s}\int_{0}^{\infty}e^{-(y_{i}-\mu \sqrt{s})^2}dy_{i}\Big]ds.\nonumber\\
&\leq & \frac{2^{N}\pi^{\frac{N}{2}}\bar{U}_{\mu}(x,t)}{\pi^{\frac{N}{2}}}\int_{0}^{\infty}e^{-(1-N\mu^{2})s}ds\nonumber\\
&=& \frac{2^{N}}{1-N\mu^2}\bar{U}_{\mu}(x,t).
\end{eqnarray}
Since $V(\cdot,\cdot;u)\leq V(\cdot,\cdot;\bar{U}_{\mu})$, hence inequality \eqref{general-spraeding-eq2} follows from \eqref{general-spraeding-eq5}.

\smallskip

For every $x\in\R^N$, $t>0$ and $i=1,\cdots,N$, we have that
\begin{align}\label{general-spraeding-eq6}
\partial_{x_{i}}V(x,t)&=\int_{0}^{\infty}\int_{\R^N}\frac{(y_{i}-x_{i})}{2s(4\pi s)^{\frac{N}{2}}}e^{-s}e^{-\frac{|x-y|^2}{4s}}u(y,t)dyds\nonumber\\
&=\frac{1}{\pi^{\frac{N}{2}}}\int_{0}^{\infty}\int_{\R^N}\frac{y_{i}}{\sqrt{s}}e^{-s}e^{-|y|^2}u(x+2\sqrt{s}y,t)dyds\nonumber\\
&\le \frac{1}{\pi^{\frac{N}{2}}}\int_{0}^{\infty}\int_{\R^N}\frac{|y_{i}|}{\sqrt{s}}e^{-s}e^{-|y|^2}\bar{U}_{\mu}(x+2\sqrt{s}y,t)dyds.
\end{align}
Observe that
\begin{align}\label{general-spraeding-eq7}
\int_{0}^{\infty}\int_{\R^N}\frac{|y_{i}|}{\sqrt{s}}e^{-s}e^{-|y|^2}\bar{U}_{\mu}(x+2\sqrt{s}y,t)dyds &\leq \bar{U}_{\mu}(x,t)\int_{0}^{\infty}\int_{\R^N}\frac{|y_{i}|}{\sqrt{s}}e^{-s}e^{-|y|^2}e^{2\mu\sqrt{s}|y|}dyds \nonumber \\
&\leq  \bar{U}_{\mu}(x,t)\int_{0}^{\infty}\int_{\R^N}\frac{|y_{i}|}{\sqrt{s}}e^{-s}e^{-(\sum_{j=1}^{N}(|y_{j}|^2-2\mu\sqrt{s}|y_{j}|)}dyds  \nonumber\\
&=  \bar{U}_{\mu}(x,t)\int_{0}^{\infty}\int_{\R^N}\frac{|y_{i}|}{\sqrt{s}}e^{-(1-N\mu^2)s}e^{-(\sum_{j=1}^{N}(|y_{j}|-\mu\sqrt{s})^{2}}dyds  \nonumber\\
&\leq  (4\pi)^{\frac{N-1}{2}}\bar{U}_{\mu}(x,t)\int_{0}^{\infty}\int_{\R}|y_{i}|e^{-(|y_{i}|-\mu\sqrt{s})^{2}}\frac{e^{-(1-N\mu^2)s}}{\sqrt{s}}dy_{i}ds  \nonumber\\
& \le  2(4\pi)^{\frac{N-1}{2}}\bar{U}_{\mu}(x,t)\int_{0}^{\infty}\frac{e^{-(1-N\mu^2)s}}{\sqrt{s}}\Big[1+\mu\sqrt{\pi s}\Big]ds  \,\,\,\,\, ({\rm by}\,\, \eqref{general-spraeding-eq7-new})\nonumber\\
&=  2(4\pi)^{\frac{N-1}{2}}\Big(\frac{\sqrt{\pi}}{\sqrt{1-N\mu^2}}+\frac{\mu\sqrt{\pi}}{1-N\mu^2}\Big)\bar{U}_{\mu}(x,t)
\end{align}
Combining \eqref{general-spraeding-eq6} and \eqref{general-spraeding-eq7}, we obtain \eqref{general-spraeding-eq3}. Hence the claim is proved.

\smallskip

{ For every $u\in \mathcal{E}_{\mu}^{T} $, let $\bar{U}(\cdot,\cdot;u)$ be the solution of the Initial Value Problem
\begin{equation}
\begin{cases}
\partial_{t}\bar{U}=\bar{\mathcal{L}}(\bar{U}), \quad x\in\R^N,\ t>0,\cr
\bar{U}(\cdot,0;u)=u_{0}
\end{cases}
\end{equation}
where
$$
\bar{\mathcal{L}}(\bar{U}):=\Delta \bar{U}-\chi\nabla V(\cdot,\cdot;u)\cdot \nabla\bar{U}+(1-\chi V(\cdot,\cdot;u)-(1-\chi)\bar{U})\bar{U}.
$$
}

\smallskip

\noindent {\bf Claim 2.} {\it Let $\mu^{*}_N\in (0, \frac{1}{\sqrt{N}})$ satisfy
\begin{equation}\label{general -equ-of-c*}
\frac{2^{N}\sqrt{N}\mu^*_N(\mu^*_N +\sqrt{1-N\mu_N^{*2}}) }{1-N\mu_N^{*2}}=\frac{1-\chi}{\chi} .
\end{equation}
Then for every  $u\in\mathcal{E}_{\mu^*_N}^{T}(u_0)$ we have that
\begin{equation}\label{general-spraeding-eq8}
0\leq \bar{U}(\cdot,\cdot;u)\leq \bar{U}_{\mu^*_N}(\cdot,\cdot).
\end{equation}
}

\medskip
The claim can be proved by the arguments similar to those in Lemma \ref{Spread-speed-sup-sol}. In the following, we provide some
indication of the proof.

Note that the function $(0\ , \frac{1}{\sqrt{N}})\ni \mu\mapsto \frac{2^{N}\sqrt{N}\mu(\mu +\sqrt{1-N\mu^{2}}) }{1-N\mu^{2}} $ is strictly increasing,  continuous and satisfies
$$\lim_{\mu\to 0^+}\frac{2^{N}\sqrt{N}\mu(\mu +\sqrt{1-N\mu^{2}}) }{1-N\mu^{2}}=0\quad \text{and} \quad \lim_{\mu\to \frac{1}{\sqrt{N}}^-}\frac{2^{N}\sqrt{N}\mu(\mu +\sqrt{1-N\mu^{2}}) }{1-N\mu^{2}}=\infty.$$
Hence the intermediate value theorem gives the existence of $\mu^*_N$, which is unique. Next, since $u_{0}\geq 0$, comparison principle for parabolic equations implies that $\bar{U}(\cdot,\cdot;u)\geq 0$.
Observe that
$$
\partial_{t}\bar{U}(\cdot,\cdot;u)\leq \Delta \bar{U}(\cdot,\cdot;u)- \chi\nabla V(\cdot,\cdot;u)\cdot \nabla \bar{U}(\cdot,\cdot;u) +(1-(1-\chi)\bar{U}(\cdot,\cdot;u))\bar{U}(\cdot,\cdot;u).
$$
Thus, comparison principle implies that  $\bar{U}(\cdot,t;u)\leq \tilde {U}(t;\|u_0\|_{\infty})$, where $\tilde U(t;\|u_0\|_\infty)$ is
the solution of \eqref{new-ode-eq} with $\tilde U(0;\|u_0\|_\infty)=\|u_0\|_\infty$. Hence
\begin{equation}\label{general-spraeding-eq9}
\bar{U}(0,t;u)\leq \tilde U(t;\|u_0\|_{\infty})\leq \max\{\frac{1}{1-\chi},\|u_0\|_{\infty}\}\leq \bar{U}_{\mu^*}(0,t) \ \forall\ t\geq 0.
\end{equation}
If we restrict $\bar{U}_{\mu^*}$ on $\R^{N}\setminus\{0\}$, using \eqref{general-spreading-eq1},  we obtain that
\begin{eqnarray}\label{general-spraeding-eq10}
\partial_{t}\bar{U}_{\mu^*_N}-\bar{\mathcal{L}}(\bar{U}_{\mu^*_N})& = &\frac{\mu(N-1)}{|x|}\bar{U}_{\mu^*_N}+\mu\chi <\nabla V(\cdot,\cdot;),\frac{x}{|x|}>\bar{U}_{\mu^*_N}+(\chi V(\cdot,\cdot;u)+(1-\chi)\bar{U}_{\mu^*_N})\bar{U}_{\mu^*_N}\nonumber\\
&\geq & -\mu\chi |\nabla V(\cdot,\cdot;u)|\bar{U}_{\mu^*_N}  +(\chi V(\cdot,\cdot;u)+(1-\chi)\bar{U}_{\mu^*_N})\bar{U}_{\mu^*_N}\nonumber\\
&\geq & -\frac{2^{N}\sqrt{N}\mu^*_N\chi(\mu^*_N +\sqrt{1-N(\mu^{*}_N)^2}}{1-N\mu^{*2}})\Big[\bar{U}_{\mu^*_N}\Big]^{2}  +(1-\chi)\Big[\bar{U}_{\mu^*_N}\Big]^{2} \nonumber\\
& =& \chi \underbrace{\Big[\frac{2^{N}\sqrt{N}\mu^*_N(\mu^*_N
+\sqrt{1-N(\mu^{*}_N)^2}) }{1-N(\mu^{*}_N)^2}-\frac{1-\chi}{\chi}
\Big]}_{=0}\Big[\bar{U}_{\mu^*_N}\Big]^{2}.
\end{eqnarray}
Combining inequalities \eqref{general-spraeding-eq9}, \eqref{general-spraeding-eq10} with the fact that $u_0\leq \bar U_{\mu^*_N}(\cdot,0)$, thus comparison principle for parabolic equations implies that $\bar{U}(\cdot,\cdot;u)\leq \bar U_{\mu^*_N}(\cdot,\cdot)$, which complete the proof of the claim.

\medskip

\noindent {\bf Claim 3.} {\it  For any given $u_0\in C_c^+(\R^N)$ and  $T>0$,  $u(\cdot,\cdot;u_0)\in
\mathcal{E}^T_{\mu}(u_{0})$ with $\mu=\mu^*_N$.
}

\medskip

The claim can be proved by the arguments similar to those in Lemma
\ref{spraeding-speed-important-lm}. We provide some indication of
the proof in the following. We put $\mu=\mu^*_N$ in the following.

Consider the normed linear space $\mathcal{E}^{T}=C^{b}_{\rm unif}(\R^N\times[0,T])$ endowed with the norm
 $$
 \|u\|_{*,T}=\sum_{n=1}^{\infty}\|u\|_{L^{\infty}(\bar{B}(0,N))\times[0\ ,\ T]}
 $$
 where $\bar{B}(0,n)$ denotes the closed ball centered at the origin in $\R^N$ with radius $n$. For every $u\in\mathcal{E}_{\mu}^{T}(u_{0})$ we have that $\|u\|_{\ast,T}\leq M e^{\mu c_{u}T}. $
Hence $\mathcal{E}_{\mu}^{T}(u_{0})$ is a bounded convex subset of $\mathcal{E}^{T}$. Since the convergence in $\mathcal{E}^{T}$ implies the pointwise convergence, then $\mathcal{E}_{\mu}^{T}(u_{0})$ is a closed, bounded, and convex subset of $\mathcal{E}^T$. Furthermore, a sequence of functions in $\mathcal{E}^T_{\mu}(u_{0})$ converges with respect to norm $\|\cdot\|_{\ast,T}$ if and only if it  converges locally uniformly on $\R^N\times[0, T]$.

By Claim 2,  every  $u\in\mathcal{E}_{\mu}^{T}(u_0)$,
$\bar{U}(\cdot,\cdot;u)\in\mathcal{E}_{\mu}^T(u_{0})$. We prove the
Claim 3 by showing that $u(\cdot,\cdot;u_0)$ is a fixed point of the
mapping $ \mathcal{E}_{\mu}^T(u_{0})\ni u\mapsto
\bar{U}(\cdot,\cdot;u)\in\mathcal{E}_{\mu}^T(u_{0})$ and divide the
proof into two steps.

\smallskip

\noindent {\bf Step 1.} In this step, we prove that the mapping
$\mathcal{E}^T_{\mu}(u_0)\ni u\mapsto \bar{U}(\cdot,\cdot;u)\in
\mathcal{E}^T_\mu(u_0)$ is compact.

Indeed, let $\{u_{n}\}_{n\geq 1}\subset \mathcal{E}_{\mu^*_N}^{T}(u_0)$ be given. For every $n\geq 1$, $\bar{U}(\cdot,\cdot;u_n)$ satisfied
$$
\begin{cases}
\partial_{t}\bar{U}(\cdot,\cdot,u_{n})=\Delta\bar{U}(\cdot,\cdot,u_{n})-\chi \nabla V(\cdot,\cdot;u_{n})\cdot\nabla \bar{U}(\cdot,\cdot;u_{n})+(1-\chi V(\cdot,\cdot;u_{n})-(1-\chi)\bar{U}(\cdot,\cdot;u_n))\bar{U}(\cdot,\cdot;u_{n})\\
\bar{U}(\cdot,0;u_n)=u_{0}
\end{cases}
$$

Taking $\{T(t)\}_{t\geq 0}$ to be the analytic semigroup generated by $A:=(\Delta-I)$ on $C^{b}_{\rm unif}(\R^N)$, the variation of constant formula and similar arguments to the one used to establish \eqref{variation -of-const} yield that for every $t\geq 0$,
\begin{eqnarray}\label{general-spraeding-eq11}
\bar{U}(\cdot,t;u_n)
& =& T(t)u_{0} -\chi\int_{0}^{t}T(t-s)\nabla\cdot \Big( \nabla V(\cdot,s;u_{n})\bar{U}(\cdot,s;u_{n})\Big)ds\nonumber\\
& &+\int_{0}^{t}T(t-s)((2-\chi u_{n}(\cdot,s)-(1-\chi)\bar{U}(\cdot,s;u_n))\bar{U}(\cdot,s;u_{n}))ds.
\end{eqnarray}
Observe that formula \eqref{general-spraeding-eq11} is equivalent to formula \eqref{spraeding-speed-important-lm} in step 1 of the proof of Lemma \ref{Spread-speed-eq07}. Same arguments used in step 1 of the proof of Lemma \ref{Spread-speed-eq07} apply to this case as well. Thus the sequence $\{\bar{U}(\cdot,\cdot;u_n)\}_{n\geq 1}$ has a subsequence that converges in $\mathcal{E}_{\mu}^{T}$. Thus the mapping $\mathcal{E}^T_\mu(u_0)\ni u\mapsto \bar{U}(\cdot,\cdot;u)\in \mathcal{E}^T_\mu(u_0)$ is compact.

\smallskip

{\bf Step 2.}  In this step, we prove that the mapping
$\mathcal{E}^T_{\mu}(u_0)\ni u\mapsto \bar{U}(\cdot,\cdot;u)\in
\mathcal{E}^T_\mu(u_0)$ is continuous.

 Using similar arguments as in the proof of Step 2 of Lemma \ref{spraeding-speed-important-lm}, we have that the mapping $\mathcal{E}^T_{\mu}(u_0)\ni u\mapsto \bar{U}(\cdot,\cdot;u)\in \mathcal{E}^T_\mu(u_0)$ is continuous.

  Schauder's fixed point theorem implies that the mapping $\mathcal{E}^T_{\mu}(u_0)\ni u\mapsto \bar{U}(\cdot,\cdot;u)\in \mathcal{E}^T_\mu(u_0)$ has a fixed point say $\bar{U}$.
   The fixed solves \eqref{IntroEq0-2-new}. Thus Therefore Theorem 1.1 in \cite{SaSh} implies that $u(\cdot,\cdot;u_0)=\bar{U}$. Hence $u(\cdot,\cdot;u_0)\in \mathcal{E}_{\mu}^{T}$.

\medskip

Now, we prove (i). First, it follows from Claim 3, that
$$
0\leq u(x,t;u_0)\leq Me^{\mu^*_N(c_{\mu^*}t-|x|)} \quad \forall\ x\in\R^N, \ t\ge 0,
$$
where $\mu^*_N$ is given by \eqref{general -equ-of-c*}. Thus for every $c>c_{\mu^*_N}$ we have that
$$
0\leq \sup_{|x|\geq c t}\leq Me^{\mu^*_N(c_{\mu^*_N}-c)t}\to 0 \quad\text{as}\quad t\to \infty.
$$
Since $c_{\mu^*_N}$ is independent of $u_{0}$ then $c_{+}^*(\chi)\leq c_{\mu^*_N}$.

Next, we know that  $u(\cdot,\cdot;u_{0})\leq \tilde
{U}(t;\|u_{0}\|_{\infty})\to \frac{1}{1-\chi}$ as $t\to \infty$,
where $\tilde U(t;\|u_0\|_\infty)$ is the solution of
\eqref{new-ode-eq} with $\tilde U(0;\|u_0\|_\infty)=\|u_0\|_\infty$.
Let $0<\varepsilon\ll 1$ be fixed. Thus there is $T_{\varepsilon}>0$
such that $\tilde U(t;\|u_{0}\|_{\infty})\leq \frac{1}{1-\chi}
+\varepsilon$ for every $t\geq T_{\varepsilon}$. We obtain that
\begin{equation}\label{Spread-speed-general-eq14}
\|u(\cdot,t;u_{0})\|_{\infty}, \,\,\, \|v(\cdot,t;u_0)\|_\infty,\,\,\,  \|v_{x_i}(\cdot,t;u_0)\|_\infty \leq  \frac{1}{1-\chi}+\varepsilon, \quad \forall t\geq T_{\varepsilon}.
\end{equation}
Let
$$M_{\varepsilon}=\sup_{0\le t\le T_{\varepsilon}}\|\nabla v (\cdot,t;u_0)\|_\infty.$$
Let $M>0$ be such that
$$
u_0(x)\le M e^{-|x|}.
$$

For any given $\xi\in S^{N-1}$, consider
\begin{equation}
\label{new-aux-eq5}
u_t=\Delta u-\chi M_\epsilon  \xi\cdot\nabla u +u(1-(1-\chi) u),\quad x\in\R^N.
\end{equation}
It is not difficult to show that $u^+(x,t)=Me^{-\big(x\cdot\xi-(2+\chi M_\epsilon)t\big)}$ is a super-solution of \eqref{new-aux-eq5}. This implies that
$u^+(x,t)$ is a super-solution of
\begin{equation}
\label{new-aux-eq00}
u_t=\Delta u-\chi \nabla v(x,t;u_0)\cdot \nabla u +u(1-\chi v(x,t;u_0)-(1-\chi)u),\quad x\in\R^N
\end{equation}
on $t\in [0,T_{\varepsilon}]$ and then
$$
u(x,t;u_0)\le M e^{-\big(x\cdot\xi -(2+\chi M_\epsilon)t\big)}\quad \xi\in S^{N-1},\,\, x\in\R,\,\, t\in [0,T_{\varepsilon}].
$$

Now for any given $\xi\in S^{N-1}$, consider
\begin{equation}
\label{new-aux-eq6}
u_t=\Delta u-\Big(\frac{\sqrt{N}\chi}{1-\chi}+\varepsilon \sqrt{N} \chi\Big)\xi\cdot \nabla u +u(1-(1-\chi)u),\quad x\in\R^N.
\end{equation}
It is not difficult to show that
$\tilde u^+(x,t)=Me^{-\big(x\cdot\xi -(2+\frac{ \sqrt{N}\chi}{1-\chi}+\varepsilon \sqrt{N} \chi)(t-T_{\varepsilon})\big)} e^{(2+\chi M_\varepsilon)T_{\varepsilon}}$ is a super-solution
of \eqref{new-aux-eq6} for $t\ge T_{\varepsilon}$. This together with and \eqref{Spread-speed-general-eq14} implies that $\tilde u^+(x,t)$
is a super-solution of \eqref{new-aux-eq00} on $[T_\varepsilon,\infty)$. It then follows that
\begin{equation}
\label{new-aux-eq7}
u(x,t;u_0)\le M e^{-\big(x\cdot\xi -(2+\frac{ \sqrt{N}\chi}{1-\chi}+\varepsilon  \sqrt{N} \chi)(t-T_{\varepsilon})\big)} e^{(2+\chi M_\varepsilon)T_{\varepsilon}},\quad x\in\R,\,\,\xi\in S^{N-1},\,\,  t\ge T_{\varepsilon},
\end{equation}
which implies that
$$
c_+^*(\chi)\le 2+\frac{\sqrt{N}\chi}{1-\chi}+\varepsilon \sqrt{N}\chi.
$$
Letting $\varepsilon \to 0$, we get
$$
c_+^*(\chi)\le 2+\frac{\sqrt{N}\chi}{1-\chi}.
$$
(i) then follows.

\medskip

(ii) If $0<\chi<\frac{2}{3+\sqrt{N+1}}$, the proof of  the uniform lower bound for $c^*_{\rm low}(u_0)$ when $N=1$ also apply to the general case.
\end{proof}

\end{document}